\documentclass[oneside]{article}
\usepackage{amssymb,amsfonts,amsmath,amsthm}
\usepackage[letterpaper]{geometry}
\usepackage{hyperref}
\usepackage{bm}
\usepackage{graphicx}
\usepackage[ruled,linesnumbered]{algorithm2e}
\usepackage{xcolor}
\usepackage{booktabs}

\newtheorem{lemma}{Lemma}
\newtheorem{theorem}{Theorem}
\newtheorem{definition}{Definition}
\newtheorem{proposition}{Proposition}
\newtheorem{remark}{Remark}
\newtheorem{example}{Example}

\def\R{\mathbb{R}}

\def\d{{\rm d}}
\def\rank{\mathop{\rm rank}}
\def\tr{\mathop{\rm tr}}

\def\diag{\mathop{\rm diag}}

\def\B#1{\left\{#1\right\}}
\def\norm#1{\left\lVert#1\right\rVert}

\allowdisplaybreaks
\title{Nonnegative Low-Rank Matrix Correction under an Orthogonality Constraint in Conservative Vlasov Simulations}
\author{
    Yue Wu\thanks{Division of Applied Mathematics, Brown University, 182 George Street, Providence, RI 02912 (\texttt{yue\_wu3@brown.edu}).},
    Stephen Becker\thanks{Department of Applied Mathematics, University of Colorado Boulder, 526 UCB, Boulder, CO 80309 (\texttt{stephen.becker@colorado.edu}).},
    Jingmei Qiu\thanks{Department of Mathematical Sciences, University of Delaware, 501 Ewing Hall, Newark, DE 19716 (\texttt{jingqiu@udel.edu}).},
    Xiangxiong Zhang\thanks{Corresponding author. Department of Mathematics, Purdue University, 150 N.\ University Street, West Lafayette, IN 47907 (\texttt{zhan1966@purdue.edu}). This work is partially supported by NSF DMS-2208518.}
}
\date{\today}
\begin{document}
\pagestyle{plain}
\maketitle

\begin{abstract}
In low-rank numerical methods for Vlasov dynamics, the SVD-type truncation procedure may introduce negative entries into the numerical solution. Such negative values are unphysical because the solution is a probability distribution function. We design optimization-based post-processing algorithms to recover nonnegativity while preserving the macroscopic quantities (density, momentum, and energy) pointwise. The preservation of the macroscopic quantities is written as an orthogonality constraint on the correction term. For a convex formulation based on squared nuclear norm minimization, we show that the proximal operator with the orthogonality constraint is characterized by an implicit singular value thresholding equation, and the threshold can be computed efficiently by bisection. Based on this result, we develop five algorithms for the convex formulation: Douglas--Rachford splitting, restarted dual FISTA, restarted dual accelerated gradient descent, dual PR+ conjugate gradient, and dual L-BFGS. We also consider a non-convex formulation with an explicit rank constraint and develop a tangent-space accelerated alternating projection algorithm that only requires a \(2r \times 2r\) SVD per iteration. Numerical results for a Landau damping test case show that the proposed algorithms give comparable correction quality. Among them, the tangent-space accelerated alternating projection is the most cost-efficient, increasingly so as the problem size grows. We further demonstrate the correction as a positivity limiter inside a time-dependent conservative low-rank Vlasov solver, where it removes the negativity introduced by the SVD-type truncation while preserving the conserved mass, momentum, and energy.
\end{abstract}


\smallskip


\bigskip


\section{Introduction}

Numerical simulation of the Vlasov--Poisson (VP) system is fundamental to understanding the complex dynamics of collisionless plasmas and has a wide range of applications in science and engineering, such as fusion energy and space weather prediction. The VP system describes the evolution of the probability distribution function \(f(\bm x,\bm v,t)\) in phase space, which by definition must be nonnegative. Its velocity moments, including mass density, momentum density, and energy density, satisfy a system of macroscopic conservation laws. Preserving both nonnegativity and conservation at the discrete level is essential for producing physically meaningful solutions and for the nonlinear stability of the numerical scheme, especially in demanding regimes such as low density and pressure.

The high dimensionality of phase space poses a major computational challenge for deterministic VP solvers. Recently, low-rank and adaptive-rank tensor methods \cite{einkemmer2025review} have emerged as promising approaches for mitigating the curse of dimensionality by adaptively exploiting low-rank structure in the solution. Broadly, these methods can be divided into two classes. The first is the dynamical low-rank approach, which evolves the low-rank factors by projecting the governing equation onto the tangent space of the low-rank manifold \cite{koch2007dynamical}. Robust and rank-adaptive time integrators have been developed within this framework \cite{lubich2014projector, ceruti2022unconventional, ceruti2022rankadaptive, dektor2025interpolatory}. The second is the step-and-truncate approach, in which the solution is first advanced in an enriched approximation space and subsequently compressed to control the rank. Within a Galerkin framework, the method proposed in \cite{guo2022low} dynamically constructs an adaptive low-rank solution basis by adding new basis functions generated from the discretized PDE and removing redundant components through SVD-type truncation. Within a collocation framework \cite{zheng2025semi, zheng2025semihighD}, selected rows and columns of the solution matrix are adaptively identified, allowing the solution to be updated through CUR-type matrix interpolation without forming the full high-dimensional solution.
An important challenge for both classes of low-rank methods is the preservation of the physical conservation laws. Conservative SVD truncation procedures were developed in \cite{einkemmer2021mass, guo2022conservative} to preserve local mass and momentum. Recently, the Local Macroscopic Conservative (LoMaC) low-rank tensor methods developed in \cite{guo2022local, guo2024local} preserve mass, momentum, and energy locally at the discrete level. Their central idea is to evolve the macroscopic conservation laws alongside the kinetic equation, construct a reference subspace from the resulting macroscopic densities, and project the low-rank kinetic solution onto this subspace through a conservative orthogonal projection.

A remaining issue for the low-rank methods in \cite{guo2022conservative, guo2022local} is the lack of nonnegativity in the numerical solution. Since the distribution function \(f\) represents a probability density, it must be nonnegative everywhere. However, SVD-type low-rank approximations of a nonnegative matrix can introduce negative entries, and this problem is further compounded by the time-stepping procedure. Within the low-rank framework, positivity can be enforced by evolving a square-root factorization $f = g^2$, so that the represented solution is nonnegative by construction \cite{ye2024quantized}. This reformulation is tied to the structure of the transport operator and does not extend to second-order operators such as a Fokker--Planck collision term, and a nonlinear transformation of a low-rank function need not remain low rank. We instead leave the low-rank solver unchanged and recover nonnegativity by a minimal post-processing correction, which is agnostic to the underlying operator. Enforcing nonnegativity (or more generally, preserving the invariant domain of the underlying PDE) in high-order numerical methods has been an active area of research. We refer to \cite{wu2025idp} for a comprehensive survey. The main approaches include polynomial limiters based on convex decomposition of cell averages \cite{perthame1996positivity, zhang2010positivity}, flux-corrected transport (FCT) methods, and more recently, optimization-based limiters. In the optimization-based approach, the limiter seeks a minimal modification of the numerical solution while enforcing conservation and positivity (or bounds) as constraints. Such optimization-based methods have been developed for enforcing nonnegativity of polynomial approximations \cite{chen2025nonneg} and for invariant-domain-preserving limiters in gas dynamics \cite{liu2026optlimiter}, where first-order operator splitting methods such as the Douglas--Rachford splitting \cite{lions1979splitting} are employed to efficiently solve the resulting constrained minimization problems. 
At the representation level, \cite{tang2025variational} proposes a nonnegative tensor train optimization framework for high-dimensional distribution tensors, whose complexity scales linearly with both the dimension \(d\) and the number of degrees of freedom per dimension \(n\).  

In this work, we design optimization-based post-processing algorithms to recover the nonnegativity of the low-rank solution produced by \cite{guo2022conservative, guo2022local}, while preserving the macroscopic quantities (density, momentum, and energy) pointwise. The problem is to find a correction term \(\bm X\) such that \(\bm A + \bm X \geqslant \bm 0\) and \(\bm X\bm B = \bm 0\), where \(\bm A\) is the possibly negative scaled solution matrix and \(\bm B\) encodes the orthogonality constraint associated with the first three velocity moments. We further split \(\bm A = \bm A_1 + \bm A_2\), where \(\bm A_1\) is a rank-3 part that carries the macroscopic quantities and \(\bm A_2\) is the remainder. To make subsequent computations efficient, it is desirable for the correction $\bm X$ to be low-rank as well. A standard approach is to use nuclear norm minimization \cite{cai2010singular} as a convex relaxation of rank minimization, which leads to the formulation
\begin{equation}
\label{F1-1 (non-homo)}
    \begin{aligned}
        & \min_{\bm X} \norm{\bm X}_{\ast} + \frac{a}{2}\norm{\bm X - \bm A_2}_{F}^2,  \\
        & \text{subject to:\quad} \bm X \geqslant -\bm A_1,\ \bm X\bm B = \bm 0,
    \end{aligned}
\end{equation}
where \(\norm{\cdot}_{\ast}\) denotes the nuclear norm (sum of singular values) and \(a > 0\) is a penalty parameter that controls the trade-off between low rank and approximation error, and the corrected solution is $\bm A_1 + \bm X$.  
However, this formulation has two drawbacks. It is not consistent, since if \(\bm A\) is already nonnegative, the solution is not necessarily the identity correction. It is also not homogeneous, since scaling the input does not scale the output. To address these two issues, we reinterpret \(\bm X\) as a correction term added to \(\bm A\) rather than an approximation to \(\bm A_2\), so the Frobenius term penalizes the magnitude of the correction itself rather than an approximation error. We also replace the nuclear norm by its square, which makes the objective homogeneous of degree two. This leads to the formulation
\begin{equation}
\label{F2-1 (homo)}
    \begin{aligned}
        & \min_{\bm X} \norm{\bm X}_{\ast}^2 + \frac{a}{2}\norm{\bm X}_{F}^2,  \\
        & \text{subject to:\quad} \bm X \geqslant -\bm A,\ \bm X\bm B = \bm 0,
    \end{aligned}
\end{equation}
where \(\bm X\) is now a correction term and the corrected solution is \(\bm A + \bm X\). This formulation is both consistent and homogeneous. One major computational challenge is to evaluate the proximal operator of the squared nuclear norm under the orthogonality constraint \(\bm X\bm B = \bm 0\). For the standard nuclear norm in \eqref{F1-1 (non-homo)}, the proximal operator is the singular value soft thresholding operator \cite{cai2010singular}.  For the squared nuclear norm in \eqref{F2-1 (homo)}, the proximal operator can still be expressed as singular value soft thresholding, but the threshold is now implicit: it depends on the nuclear norm of the unknown minimizer itself. We derive this characterization in Theorem~\ref{thm:proximal} and show that the implicit threshold can be computed by a bisection search that requires only \(\mathcal{O}(\log(\min(m,n)))\) singular value comparisons for an \(m \times n\) matrix. 

We consider two formulations for the nonnegative correction problem. For the convex problem \eqref{F2-1 (homo)}, we develop five algorithms: Douglas--Rachford splitting \cite{lions1979splitting} on the primal problem, restarted dual FISTA \cite{doi:10.1137/080716542}, restarted dual accelerated gradient descent, dual Polak--Ribi\`{e}re plus conjugate gradient, and dual L-BFGS \cite{liu1989limited}. These algorithms converge to a global minimizer. We also consider a non-convex formulation with an explicit rank constraint on the correction term. For this formulation, we develop a tangent-space accelerated alternating projection (TAP) algorithm inspired by \cite{song2020tangent}. The key idea is to work on the tangent space of the low-rank manifold so that each iteration only requires a \(2r \times 2r\) SVD for a rank $r$ solution instead of a full SVD. Beyond the fixed correction problem, we demonstrate the correction as a positivity limiter inside the time-dependent LoMaC conservative low-rank scheme of \cite{guo2022local}. Applied periodically as the solution is advanced, it removes the negativity introduced by the SVD-type truncation while leaving the conserved moments untouched, so that mass, momentum, and energy remain conserved.

The rest of the paper is organized as follows. In Section~2, we describe the problem background, including the low-rank solution structure for Vlasov dynamics and the conservative orthogonal projection from \cite{guo2022conservative, guo2022local}. In Section~3, we formulate the nonnegative low-rank matrix correction problem with orthogonality constraint and present three optimization formulations. In Section~4, we develop the numerical algorithms for both the convex and non-convex formulations. In Section~5, we present numerical results on a Landau damping test case that demonstrate the effectiveness and compare the performance of all proposed algorithms, study their empirical cost scaling with problem size, and demonstrate the correction as a positivity limiter in a time-dependent conservative low-rank Vlasov solver. We conclude in Section~6.

\section{Low-rank Vlasov dynamics and the orthogonality constraint}

\subsection{Low-rank Vlasov solutions}

We consider a 1D1V Vlasov system \cite{guo2022low}
\begin{align}
    &\frac{\partial f}{\partial t} + v \cdot \nabla_{x} f - E(x, t) \cdot \nabla_{v} f = 0, \quad \Omega_{x} \times \Omega_{v} \times (0,T], \nonumber\\
    &-\Delta_{x} \phi = \rho(x, t), \quad
    \rho(x, t) = \rho_0 - \int_{\mathbb{R}^{d_v}} f(x, v, t) \,\mathrm{d}v,
    \quad \Omega_{x} \times (0,T], 
 \nonumber\\
    &E(x,t) = -\nabla_{x} \phi, 
    \quad \Omega_{x} \times (0,T], 
\nonumber
\end{align}
where $f$ is the probability density function of charged particles, with a corresponding charge density $\rho$, which generates a self-consistent electrostatic potential $\phi$ via the Poisson equation. The Vlasov equation is discretized on the phase space \([x_{\min}, x_{\max}]\times[-v_{\max}, v_{\max}]\) with an equidistant \(N_{x}\times N_{v}\) Cartesian grid
\begin{eqnarray*}
    & x_{\text{grid}}: x_{\min} = x_1 < \cdots < x_{N_{x}} = x_{\max}, \\
    & v_{\text{grid}}: -v_{\max} = v_1 < \cdots < v_{N_{v}} = v_{\max}.
\end{eqnarray*}
We denote the mesh size in the \(v\)-direction by \(h_v\). The numerical solution \(\bm F \in \R^{N_{x}\times N_{v}}\), with \(\bm F_{i,j} \approx f(x_i, v_j)\) approximating the probability distribution function \(f(x,v)\), is stored in a low-rank format
\begin{equation}
    \bm F
    = \bm F_{1} + \bm F_{2}
    = \bm U_{1}\bm V_{1}^T + \bm U_{2}\bm V_{2}^T,
\end{equation}
where \(\bm F_1\) is a rank-3 term that carries the three macroscopic quantities (density, momentum, and internal energy) of \(\bm F\), and \(\bm F_2\) is the remainder. This decomposition is the key structure underlying the conservative low-rank tensor methods in \cite{guo2022conservative, guo2022local}.

The three macroscopic quantities are defined as velocity moments of \(f(x,v)\):
\begin{subequations}
    \begin{eqnarray}
        & \rho(x) = \int f(x,v)\d v, \\
        & m(x) = \int v f(x,v)\d v, \\
        & e(x) = \int v^2 f(x,v)\d v.
    \end{eqnarray}
\end{subequations}
Let \(\bm 1 = (1,\dots,1) \in \R^{1 \times N_v}\), \(\bm v = (v_1,\dots,v_{N_v}) \in \R^{1 \times N_v}\), and \(\bm v^{\circ 2} = (v_1^2,\dots,v_{N_v}^2) \in \R^{1 \times N_v}\). The discrete macroscopic quantities are computed as
\begin{subequations}
    \begin{eqnarray}
        & \bm \rho = h_v \bm F \bm 1^T = h_v \bm F_1 \bm 1^T, \\
        & \bm m = h_v \bm F \bm v^T = h_v \bm F_1 \bm v^T, \\
        & \bm e = h_v \bm F (\bm v^{\circ 2})^T = h_v \bm F_1 (\bm v^{\circ 2})^T,
    \end{eqnarray}
\end{subequations}
where the second equalities hold by the requirement that \(\bm F_1\) carries all three macroscopic quantities, i.e., \(\bm F_2\) has zero velocity moments.

\subsection{Conservative orthogonal decomposition}

Following \cite{guo2022conservative}, we construct the decomposition \(\bm F = \bm F_1 + \bm F_2\) using a weighted inner product in the velocity space. Let
\begin{equation}
    \bm R
    =\begin{pmatrix}
        \bm 1 \\
        \bm v \\
        \bm v^{\circ 2}
    \end{pmatrix}
    \in\R^{3\times N_v}
\end{equation}
be the matrix whose rows span the space of velocity moments, and let
\begin{equation}
    \bm S = \diag\left(w(v_{1}), w(v_{2}), \cdots, w(v_{N_v})\right) \in\R^{N_v \times N_v}
\end{equation}
be a diagonal weight matrix with pointwise-positive weight function \(w(v) > 0\). The requirement that \(\bm F_2\) has zero velocity moments is expressed as
\begin{equation}
    \bm F_2 \bm R^T = \bm 0, \qquad \text{or equivalently,} \qquad \bm F_1 \bm R^T = \bm F \bm R^T.
\end{equation}
Additionally, we require \({\rm rank}(\bm F_1) = 3\) and that the row space of \(\bm F_1\) lies in the column space of \(\bm R\bm S\). This gives the system
\begin{equation}
    \begin{cases}
        \bm F_1 = \bm L \bm R \bm S, \\
        \bm F_1 \bm R^T = \bm F \bm R^T,
    \end{cases}
\end{equation}
where \(\bm L \in \R^{N_x\times 3}\). To solve for \(\bm F_1\) explicitly, we compute a QR decomposition
\begin{equation}
    \sqrt{\bm S}\,\bm R^T = \bm B\widetilde{\bm R},
\end{equation}
where \(\bm B \in \R^{N_v \times 3}\) has orthonormal columns and \(\widetilde{\bm R} \in {\rm GL}_{3}(\R)\). Then \(\bm F_1\) is given by
\begin{equation}\label{F1-formula}
    \bm F_1 = \bm F \sqrt{\bm S}^{-1} \bm B\bm B^{T}\sqrt{\bm S}.
\end{equation}
Thus, \(\bm F_1\) is the orthogonal projection of \(\bm F\) (with respect to the weighted inner product defined by \(\bm S\)) onto the three-dimensional subspace spanned by the velocity moment basis functions \(\bm 1, \bm v, \bm v^{\circ 2}\).

\subsection{Rescaled problem and the orthogonality constraint}\label{sec:rescaled}

To formulate the correction problem, we introduce the rescaled matrices
\begin{subequations}\label{rescaled-matrices}
    \begin{eqnarray}
        & \bm A_1 = \bm F_1\sqrt{\bm S}^{-1},  \\
        & \bm A_2 = \bm F_2\sqrt{\bm S}^{-1},  \\
        & \bm A = \bm A_1 + \bm A_2 = \bm F\sqrt{\bm S}^{-1}.
    \end{eqnarray}
\end{subequations}
The rescaling converts the weighted inner product to the standard Frobenius inner product, so that the orthogonal projection \eqref{F1-formula} takes the simple form \(\bm A_1 = \bm A \bm B\bm B^T\). Note that \(\bm A_1\) and \(\bm A_2\) inherit the low-rank structure of \(\bm F_1\) and \(\bm F_2\), and the nonnegativity of \(\bm F\) is equivalent to the nonnegativity of \(\bm A\) (since \(\bm S\) has positive diagonal entries).

A remaining issue for adaptive rank approaches is that the matrix \(\bm A\) (or equivalently \(\bm F\)) may contain negative entries. We seek a correction term \(\bm X\) such that \(\bm A + \bm X \geqslant 0\). We also require the corrected solution \(\bm A + \bm X\) to have the same macroscopic quantities as \(\bm A\). Since \(\bm A_1 = \bm A \bm B\bm B^T\) carries the macroscopic information, the requirement that the correction \(\bm X\) does not alter the macroscopic quantities is equivalent to
\begin{equation}\label{orthogonality-constraint}
    \bm X \bm B = \bm 0.
\end{equation}
This is the orthogonality constraint: the correction term \(\bm X\) must lie in the orthogonal complement of the column space of \(\bm B\). In the following, we use \(\bm B \in \R^{N_v \times 3}\) to denote the matrix with orthonormal columns defined in this section.

\section{Nonnegative low-rank matrix correction: optimization formulations}

With the setup from the previous section, the goal is to find a correction term \(\bm X\) such that the corrected solution \(\bm A + \bm X\) is nonnegative and preserves the macroscopic quantities. The constraints are
\[
    \bm A + \bm X \geqslant 0, \qquad \bm X \bm B = \bm 0.
\]
Ideally, the correction \(\bm X\) should have low rank so that the corrected solution retains a compact low-rank representation. We present three optimization formulations for this problem.

\subsection{The approximation formulation}\label{sec:approx-formulation}

We first seek a nonnegative low-rank approximation \(\bm X\) to the remainder \(\bm A_2\), so that the corrected solution \(\bm A_1 + \bm X\) is nonnegative and preserves the macroscopic quantities. The ideal problem is
\begin{equation}\label{low-rank approximation}
    \min_{\bm X} \rank(\bm X) \text{\ \ and\ \ }\norm{\bm X - \bm A_2}_{F},
    \quad \text{subject to} \quad \bm X + \bm A_1 \geqslant 0,\ \bm X\bm B = \bm 0.
\end{equation}
However, rank minimization is NP-hard in general. A standard convex relaxation technique is to replace the rank function by the nuclear norm \cite{cai2010singular}.

\begin{definition}[nuclear norm]
    For any matrix \(\bm A \in \R^{m\times n}\), its nuclear norm is defined as
    \begin{equation}
        \norm{\bm A}_{\ast} = \sum_{k=1}^{\min\B{m,n}} \sigma_{k},
    \end{equation}
    where \(\sigma_k\) is the \(k\)-th singular value of \(\bm A\).
\end{definition}

The nuclear norm is the tightest convex lower bound of the rank function on the unit spectral norm ball. Minimizing the nuclear norm therefore promotes low-rank solutions in the same way that minimizing the \(\ell^1\) norm of a vector promotes sparsity. We record a standard property that will be used in the algorithm derivation, using $\|\cdot\|_2$ to denote the spectral norm.

\begin{proposition}\label{prop:subgradient}
    Suppose the compact SVD of \(\bm X \in\R^{m\times n}\) is \(\bm X = \bm U \bm\Sigma\bm V^T\). Then the  subdifferential of \(\norm{\cdot}_{\ast}\) at \(\bm X\) is
    \begin{equation}
        \partial \norm{\bm X}_{\ast} = \B{\bm U\bm V^T + \bm W \;\big|\; \bm U^T \bm W = \bm 0,\ \bm W \bm V = \bm 0,\ \norm{\bm W}_{2}\leqslant 1}.
    \end{equation}
\end{proposition}

Using the nuclear norm as a convex surrogate for the rank, and choosing the Frobenius norm for the approximation error, we obtain the following convex optimization problem:
\begin{equation*}
    \begin{aligned}
        & \min_{\bm X} \norm{\bm X}_{\ast} + \frac{a}{2}\norm{\bm X - \bm A_2}_{F}^2,  \\
        & \text{subject to:\quad} \bm X \geqslant -\bm A_1,\ \bm X\bm B = \bm 0,
    \end{aligned}
\end{equation*}
where \(a > 0\) is a parameter that controls the trade-off between low rank and approximation accuracy. This is precisely the formulation in \eqref{F1-1 (non-homo)}.

\subsection{The correction formulation (convex)}\label{sec:correction-formulation}

The formulation \eqref{F1-1 (non-homo)} has two drawbacks. First, it is not consistent: if \(\bm A\) is already nonnegative, the minimizer \(\bm X\) of \eqref{F1-1 (non-homo)} is not necessarily \(\bm A_2\), meaning that the post-processing may alter a solution that needs no correction. This can be seen from the following example.

\begin{example}\label{ex:inconsistency}
    Drop the orthogonality constraint and consider the diagonal case: \(\bm A_1 = \bm 0\) and \(\bm A_2 = \diag\left(a_1,\cdots,a_n\right)\) with \(a_i \geqslant 0\). Then \(\norm{\bm X}_{\ast} = \sum_i |x_i|\) reduces to the \(\ell^1\) norm of the diagonal, and singular value thresholding acts entrywise, so the minimizer of \eqref{F1-1 (non-homo)} is \(\bm X = \diag\left(x_1,\cdots,x_n\right)\) with
    \[
        x_i = \max\B{a_i - \frac{1}{a},\ 0},
    \]
    which differs from \(\bm A_2\) unless \(a_i \geqslant 1/a\) for all \(i\).
\end{example}

Second, the formulation is not homogeneous: if the input is scaled by \(\alpha > 0\), the minimizer does not scale by the same factor. This can also be verified through Example~\ref{ex:inconsistency}.

To address these two issues, we seek a \emph{correction term} \(\bm X\) to be added to \(\bm A\), rather than an approximation to \(\bm A_2\). We also replace the nuclear norm by its square, which makes the objective homogeneous of degree two. This leads to the formulation
\begin{equation*}
    \begin{aligned}
        & \min_{\bm X} \norm{\bm X}_{\ast}^2 + \frac{a}{2}\norm{\bm X}_{F}^2,  \\
        & \text{subject to:\quad} \bm X \geqslant -\bm A,\ \bm X\bm B = \bm 0.
    \end{aligned}
\end{equation*}
which is the formulation \eqref{F2-1 (homo)}.
Here \(\bm X\) is the correction term and the corrected solution is \(\bm A + \bm X\). The constraint \(\bm X \geqslant -\bm A\) ensures nonnegativity of the corrected solution, and the constraint \(\bm X\bm B = \bm 0\) ensures preservation of the macroscopic quantities. The squared nuclear norm \(\norm{\bm X}_{\ast}^2\) promotes low rank (since \(\rank(\bm X + \bm Y) \leqslant \rank(\bm X) + \rank(\bm Y)\), a small correction should have low rank), while the Frobenius norm penalty \(\frac{a}{2}\norm{\bm X}_{F}^2\) penalizes large corrections. We refer to \eqref{F2-1 (homo)} as the convex formulation in the rest of the paper.

This formulation is consistent: if \(\bm A \geqslant 0\), then \(\bm X = \bm 0\) is clearly the minimizer. It is also homogeneous: if the input is \(\alpha\bm A\), the minimizer is \(\alpha\bm X\).

\begin{remark}
    In practice, we have not observed significant differences in the numerical results between the formulations \eqref{F1-1 (non-homo)} and \eqref{F2-1 (homo)} when the parameter \(a\) is well chosen.
\end{remark}

\begin{remark}[Non-emptiness of the admissible set]
    The admissible set \(\B{\bm X \in \R^{m \times n} \big| \bm X \geqslant -\bm A,\ \bm X \bm B = \bm 0}\) may be empty. It is nonempty if and only if there exists a nonnegative matrix \(\widetilde{\bm A} \in \R^{m \times n}\) with \(\widetilde{\bm A} \bm B\bm B^T = \bm A \bm B\bm B^T\), i.e., the macroscopic part of \(\bm A\) coincides with the macroscopic part of some nonnegative distribution. In the Vlasov dynamics setting, this condition is satisfied whenever the macroscopic moments \((\rho, m, e)\) are realizable, i.e., there exists a nonnegative distribution with these moments. For the first three moments, a necessary and sufficient condition is \(\rho > 0\) and \(e\rho \geqslant m^2\).
\end{remark}

\subsection{The rank-constrained formulation (non-convex)}\label{sec:nonconvex-formulation}

In some applications, it is desirable to explicitly control the rank of the correction term so that the corrected solution can be stored efficiently in a low-rank format. For a given target rank \(r\), we consider the non-convex formulation
\begin{equation}\label{F2-2}
    \begin{aligned}
        & \min_{\bm X}\norm{\bm X}_{F}^2, \\
        & \text{subject to:\quad} \rank(\bm X)\leqslant r,\ \bm X \geqslant -\bm A,\ \bm X\bm B = \bm 0.
    \end{aligned}
\end{equation}
Unlike the convex formulations \eqref{F1-1 (non-homo)} and \eqref{F2-1 (homo)}, the rank constraint makes the feasible set non-convex, so global optimality cannot be guaranteed. However, this formulation has the computational advantage that the iterates can be maintained in a rank-\(r\) factored form throughout the algorithm, as we will see in Section~4. We refer to \eqref{F2-2} as the non-convex formulation in the rest of the paper.

\section{Algorithms}\label{sec:algorithms}

In this section, we develop algorithms for both the convex formulation \eqref{F2-1 (homo)} and the non-convex formulation \eqref{F2-2}. For the convex problem, we present five algorithms in Sections~\ref{sec:DR}--\ref{sec:LBFGS}: Douglas--Rachford splitting on the primal problem, restarted dual FISTA, restarted dual accelerated gradient descent, dual PR+ conjugate gradient, and dual L-BFGS. For the non-convex problem, we develop a tangent-space accelerated alternating projection algorithm in Section~\ref{sec:TAP}. We begin with a brief review of the optimization tools used throughout this section. Readers familiar with convex optimization may skip to Section~\ref{sec:DR}.

\subsection{Convex optimization preliminaries}\label{sec:opt-prelim}

The convex formulation \eqref{F2-1 (homo)} involves minimizing a non-smooth objective (the squared nuclear norm plus a Frobenius penalty) subject to non-smooth constraints (nonnegativity and orthogonality). We briefly review the key concepts needed for the algorithms.

\paragraph{Indicator functions and subdifferentials.}
To incorporate constraints into an unconstrained framework, we use the indicator function of a convex set \(\mathcal{C}\):
\[
    I_{\mathcal{C}}(\bm X) = \begin{cases} 0 & \text{if } \bm X \in \mathcal{C}, \\ +\infty & \text{otherwise,} \end{cases}
\]
so that \(\min_{\bm X \in \mathcal{C}} \varphi(\bm X) = \min_{\bm X} \varphi(\bm X) + I_{\mathcal{C}}(\bm X)\). For a convex but non-differentiable function \(\varphi\), the gradient is replaced by the subdifferential \(\partial\varphi(\bm X_0)\), defined as the set of all \(\bm G\) satisfying \(\varphi(\bm Y) \geqslant \varphi(\bm X_0) + \langle \bm G, \bm Y - \bm X_0 \rangle\) for all \(\bm Y\), where the inner product is defined $\langle \bm Y, \bm X\rangle = \tr(\bm Y^T \bm X)$. When \(\varphi\) is differentiable, \(\partial\varphi = \{\nabla\varphi\}\). As a scalar example, \(\partial|x|(0) = [-1,1]\). The subdifferential of the nuclear norm is given in Proposition~\ref{prop:subgradient}.

\paragraph{Proximal operators.}
Many convex optimization problems can be interpreted as finding the steady state of the (sub)gradient flow
\[
    \dot{\bm X}\in -\partial\varphi(\bm X),
\]
where $\partial\varphi$ denotes the subdifferential of a proper, closed, convex function $\varphi$. When $\varphi$ is differentiable, this reduces to the familiar gradient flow
\[
    \dot{\bm X}=-\nabla\varphi(\bm X).
\]
Applying the backward Euler method with time step $\Delta t$ gives
\[
    \frac{\bm X^{k+1}-\bm X^k}{\Delta t}
    \in
    -\partial\varphi(\bm X^{k+1}),
\]
or equivalently,
\[
    \bm0
    \in
    \partial\varphi(\bm X^{k+1})
    +
    \frac{1}{\Delta t}
    (\bm X^{k+1}-\bm X^k).
\]
Remarkably, this is precisely the first-order optimality condition of the optimization problem
\begin{equation}\label{eq:prox-def}
    \bm X^{k+1}
    =
    \arg\min_{\bm X}
    \left\{
        \varphi(\bm X)
        +
        \frac{1}{2\Delta t}
        \|\bm X-\bm X^k\|_F^2
    \right\}.
\end{equation}
Identifying $\eta=1/\Delta t$ yields the proximal operator
\[
    \mathrm{prox}_{\eta^{-1}\varphi}(\bm Z)
    =
    \arg\min_{\bm X}
    \left\{
        \varphi(\bm X)
        +
        \frac{\eta}{2}
        \|\bm X-\bm Z\|_F^2
    \right\},
\]
which may therefore be viewed as a single implicit (backward Euler) step of the subgradient flow. Consequently, the proximal point iteration
\[
    \bm X^{k+1}
    =
    \mathrm{prox}_{\eta^{-1}\varphi}(\bm X^k)
\]
inherits the unconditional stability characteristic of backward Euler and converges to a minimizer of $\varphi$ under standard assumptions.
An important special case is the indicator function of a closed convex set,
\[
I_{\mathcal C}(\bm X)=
\begin{cases}
0,&\bm X\in\mathcal C,\\
+\infty,&\text{otherwise},
\end{cases}
\]
whose proximal operator reduces to the Euclidean projection,
\[
\mathrm{prox}_{\eta^{-1}I_{\mathcal C}}
=
\Pi_{\mathcal C},
\]
independent of $\eta$.

\paragraph{Splitting method.}
Many optimization problems take the composite form
\[
    \min_{\bm X}\; f(\bm X)+g(\bm X),
\]
where different terms represent different objectives or constraints. The choice of numerical method depends on the structure of these two components.
When both $f$ and $g$ are differentiable, their gradients can be combined,
\[
\nabla(f+g)=\nabla f+\nabla g,
\]
and standard gradient-based methods apply. If only one term (say $f$) is differentiable while $g$ is nonsmooth, then one may take an explicit gradient step on $f$ followed by an implicit (proximal) step on $g$, leading to the proximal-gradient method.
The more challenging situation arises when both $f$ and $g$ are nonsmooth. In principle, one could apply the proximal point method directly to the sum,
\[
\operatorname{prox}_{\eta^{-1}(f+g)},
\]
but this generally requires solving a difficult optimization problem involving both functions simultaneously. In contrast, the individual proximal operators,
\[
\operatorname{prox}_{\eta^{-1}f}
\quad\text{and}\quad
\operatorname{prox}_{\eta^{-1}g},
\]
are often available in closed form or can be computed efficiently because each function possesses a much simpler structure. Proximal splitting methods exploit this observation by replacing one difficult proximal evaluation of $f+g$ with a sequence of simpler proximal evaluations involving only $f$ and $g$. This philosophy is analogous to operator splitting methods for differential equations, where a complicated evolution operator is decomposed into simpler subproblems that can be solved individually. Douglas--Rachford splitting is one of the most successful examples of such proximal splitting methods, alternating between the proximal operators of $f$ and $g$ while still converging to a minimizer of the composite problem.
Our convex problem \eqref{F2-1 (homo)} falls precisely into this setting. The nonnegativity constraint and the orthogonality-constrained nuclear norm are both nonsmooth, making the proximal operator of their sum difficult to compute directly. However, each individual proximal operator admits an efficient closed-form solution, making Douglas--Rachford splitting a natural and effective approach.

\subsection{Douglas--Rachford splitting for the convex problem}\label{sec:DR}

\subsubsection{Splitting}

To apply Douglas--Rachford splitting, the objective should be decomposed into two convex functions whose proximal operators are individually inexpensive to evaluate. The goal is therefore not to split the objective arbitrarily, but rather to isolate the different structures so that each proximal subproblem admits either a closed-form solution or an efficient numerical algorithm.

For our problem, the two nonsmooth components arise from the nonnegativity constraint and the orthogonality-constrained nuclear norm. These two structures are naturally handled by different proximal operators, leading to the decomposition

\begin{subequations}\label{the splitting}
    \begin{eqnarray}
        & h(\bm X) = f(\bm X) + g(\bm X), \\
        & f(\bm X) = \frac{a_{1}}{2}\norm{\bm X}_{F}^{2} + I_{\B{\bm X|\bm X\geqslant -\bm A}}(\bm X), \\
        & g(\bm X) = \norm{\bm X}_{\ast}^2 + \frac{a_{2}}{2}\norm{\bm X}_{F}^2 + I_{\B{\bm X|\bm X\bm B = \bm 0}}(\bm X),
    \end{eqnarray}
\end{subequations}
where \(a_{1} + a_{2} = a\).

The quadratic regularization is split between $f$ and $g$ through the parameters $a_1$ and $a_2$. This decomposition preserves the original objective while allowing the proximal operators of both subproblems to remain well conditioned. In particular,

\begin{itemize}
\item
$f$ consists of a quadratic term together with the box constraint
$\bm X\ge -\bm A$. Its proximal operator is completely separable over the matrix entries and reduces to an elementwise projection.

\item
$g$ combines the squared nuclear norm and the linear orthogonality constraint $\bm X\bm B=\bm0$. Although this proximal operator is more involved, the orthogonality constraint has a simple linear structure, allowing a closed-form solution to be derived through the optimality conditions.
\end{itemize}

Since both $f$ and $g$ are proper, closed, and convex, Douglas--Rachford splitting is applicable. Instead of solving one difficult proximal problem involving
\[
\operatorname{prox}_{\eta^{-1}(f+g)},
\]
the algorithm alternates between the much simpler proximal operators
\[
\operatorname{prox}_{\eta^{-1}f}
\quad\text{and}\quad
\operatorname{prox}_{\eta^{-1}g},
\]
while still converging to a minimizer of the original convex problem.

\begin{algorithm}[htbp]
    \SetAlgoLined

    \KwIn{\(\bm Z_{1}\), \(\bm A\), \(\bm B\), \(a_1,a_2 \geqslant 0\), \(\eta > 0\), \(\alpha \in (0, 1]\), \(\text{maxit}\)}

    \KwOut{\(\bm X_\text{maxit}\)}

    \For{\(k=1:\text{maxit}\)}{

        \(\bm X_k \leftarrow \arg\min_{\bm X} f(\bm X) + \frac{\eta}{2}\norm{\bm X - \bm Z_{k}}_{F}^2\) \tcp*{first proximal operator}

        \(\bm Y_k \leftarrow \arg\min_{\bm Y} g(\bm Y) + \frac{\eta}{2}\norm{\bm Y - 2\bm X_{k} + \bm Z_{k}}_{F}^2\) \tcp*{second proximal operator}

        \(\bm Z_{k+1} \leftarrow \bm Z_{k} + 2\alpha\left(\bm Y_{k} - \bm X_{k}\right)\)\;
    }

    \caption{DR for convex problem}
    \label{DR for convex problem}
\end{algorithm}

The proximal operator of \(f\) has a simple closed form. Since \(f\) is a quadratic plus an indicator function for the set \(\{\bm X \mid \bm X \geqslant -\bm A\}\), the problem is separable and the proximal operator is an elementwise clipping:
\begin{equation}
    \arg\min_{\bm X} f(\bm X) + \frac{\eta}{2}\norm{\bm X - \bm Z_{k}}_{F}^2
    =\pi_{+}\!\left(\frac{\eta}{a_{1} + \eta}\bm Z_{k}\right),
\end{equation}
where
\begin{equation}
    \pi_{+}(\bm Z)_{ij} = \max\B{-A_{ij}, Z_{ij}}.
\end{equation}

\subsubsection{The proximal operator of the orthogonality-constrained squared nuclear norm}

Unlike the proximal operator of $f$, which is completely separable over the matrix entries, the proximal operator of $g$ is substantially more challenging. The difficulty arises from two sources. First, the squared nuclear norm
\[
\|\bm Y\|_*^2
=
\left(\sum_i\sigma_i(\bm Y)\right)^2
\]
couples all singular values, so the optimization problem is no longer separable. Second, the orthogonality constraint
\[
\bm Y\bm B=\bm0
\]
introduces a linear constraint on the admissible right singular vectors. Consequently, unlike the proximal operator of $f$, no elementwise solution is available.
Our strategy is therefore different. Rather than solving the minimization problem directly, we first derive its first-order optimality condition, identify its structure, and then show that this condition is equivalent to a modified singular-value thresholding problem. This ultimately leads to a closed-form characterization of the proximal operator.

The optimality condition for
\begin{equation}\label{Proximal2-0}
    \bm Y = \arg\min_{\bm Y} g(\bm Y) + \frac{\eta}{2}\norm{\bm Y - \bm M}_{F}^2
\end{equation}
is
\begin{equation}\label{Proximal2-1}
    \begin{cases}
        \bm 0 \in \partial \norm{\bm Y}_{\ast}^2 + (a_{2}+\eta)\bm Y - \eta\bm M + \B{\bm\Lambda\bm B^T \mid \bm\Lambda \in \R^{N_x \times 3}}, \\
        \bm Y\bm B = \bm 0.
    \end{cases}
\end{equation}

We first state a lemma for the subgradient of the squared nuclear norm.
\begin{lemma}\label{lem:squared-nuclear-subgradient}
    \(2\norm{\bm X}_{\ast}\partial\norm{\bm X}_{\ast} \subset \partial \norm{\bm X}_{\ast}^2\).
    \begin{proof}
        For each \(\bm Z \in \partial\norm{\bm X}_{\ast}\), we have \(\norm{\bm Y}_{\ast} - \norm{\bm X}_{\ast} \geqslant \left\langle \bm Z, \bm Y - \bm X \right\rangle\) for all \(\bm Y\). Then
        \[\begin{aligned}
            \norm{\bm Y}_{\ast}^2 - \norm{\bm X}_{\ast}^2
            - 2\norm{\bm X}_{\ast}\left\langle \bm Z, \bm Y - \bm X \right\rangle
            &\geqslant \norm{\bm Y}_{\ast}^2 - \norm{\bm X}_{\ast}^2
            - 2\norm{\bm X}_{\ast}\left(\norm{\bm Y}_{\ast} - \norm{\bm X}_{\ast}\right) \\
            &= \left(\norm{\bm Y}_{\ast} - \norm{\bm X}_{\ast}\right)^2 \geqslant 0.
        \end{aligned}\]
       Therefore, \(2\norm{\bm X}_{\ast}\bm Z \in \partial \norm{\bm X}_{\ast}^2\).
    \end{proof}
\end{lemma}

\begin{definition}
    For any matrix \(\bm O\in\R^{n\times k}\) (\(k\leqslant n\)) with orthonormal columns, we define the orthogonal projection and its complement as
    \begin{equation}
        \bm\Pi_{\bm O} = \bm O\bm O^T, \qquad
        \bm\Pi_{\bm O}^{\perp} = \bm I - \bm O\bm O^T.
    \end{equation}
\end{definition}

The next theorem characterizes the proximal operator of \(g\). It shows that the proximal operator can be expressed in terms of the singular value soft thresholding operator \(D_{\tau}(\cdot)\), which subtracts \(\tau\) from each singular value and discards the non-positive ones.

\begin{theorem}\label{thm:proximal}
    Suppose the compact SVD of the minimizer \(\bm Y\) of \eqref{Proximal2-0} is \(\bm Y = \bm U\bm\Sigma\bm V^T\). Then \(\bm Y\) satisfies
    \begin{equation}\label{Proximal2---Implicit}
        \bm Y = D_{\frac{2\norm{\bm Y}_{\ast}}{a_{2}+\eta}}\left(\frac{\eta}{a_{2}+\eta}\bm M \bm\Pi_{\bm B}^{\perp}\right),
    \end{equation}
    where \(D_{\tau}(\cdot)\) is the singular value soft thresholding operator with threshold \(\tau\).

    \begin{proof}
        Since \(g\) is strongly convex, problem \eqref{Proximal2-0} has a unique minimizer. We construct a candidate \(\bm Y\) via \eqref{Proximal2---Implicit} and verify that it satisfies the first-order optimality condition \eqref{Proximal2-1}. By Lemma~\ref{lem:squared-nuclear-subgradient} and Proposition~\ref{prop:subgradient}, \(\bm Y\) satisfies \eqref{Proximal2-1} if the following system admits a solution:
        \begin{equation}
            \begin{aligned}
                & \exists \bm Z \in\R^{N_{x}\times 3}, \bm W \in\B{\bm W\in\R^{N_x\times N_v} | \bm U^T\bm W = \bm 0, \bm W\bm V = \bm 0, \norm{\bm W}_2 \leqslant 1}: \\
                & \begin{cases}
                    \bm 0 = 2\norm{\bm\Sigma}_{\ast}\bm U\bm V^T + 2\norm{\bm\Sigma}_{\ast}\bm W + \bm Z \bm B^T + (a_{2}+\eta)\bm U\bm\Sigma\bm V^T - \eta\bm M, \\
                    \bm V^T \bm B = \bm 0, 
                \end{cases}
            \end{aligned}
        \end{equation}
        which can be rewritten as 
        \begin{equation}
            \begin{aligned}
                & \exists \bm Z \in\R^{N_{x}\times 3}, \bm W \in\B{\bm W\in\R^{N_x\times N_v} | \bm U^T\bm W = \bm 0, \bm W\bm V = \bm 0, \norm{\bm W}_2 \leqslant 1}: \\
                & \begin{cases}
                    \frac{\eta}{a_{2} + \eta}\bm M = \bm U \left(\bm\Sigma + \frac{2\norm{\bm\Sigma}_{\ast}}{a_{2} + \eta}\right) \bm V^T + \frac{2\norm{\bm\Sigma}_{\ast}}{a_{2} + \eta}\bm W + \bm Z\bm B^T, \\
                    \bm V^T \bm B = \bm 0. 
                \end{cases}
            \end{aligned}
        \end{equation}

        We now show that this system admits a solution. This proves that the candidate \(\bm Y\) constructed via \eqref{Proximal2---Implicit} satisfies the optimality condition and is therefore the unique minimizer.
        
        Let \(\tau = \frac{2\norm{\bm\Sigma}_{\ast}}{a_{2}+\eta}\), then
        \begin{equation}
            \begin{aligned}
                \frac{\eta}{a_{2} + \eta}\bm M
                &= \bm U \left(\bm\Sigma + \tau\right) \bm V^T + \tau\bm W + \bm Z\bm B^T \\
                &= \bm U \left(\bm\Sigma + \tau\right) \bm V^T + \tau\bm W(\bm I-\bm B\bm B^T) + \left(\bm Z + \tau\bm W\bm B\right)\bm B^T \\
                &= \begin{pmatrix}
                    \bm U & \widetilde{\bm U}
                \end{pmatrix}
                \begin{pmatrix}
                    \bm\Sigma + \tau & \bm 0\\
                    \bm 0 & \tau\widetilde{\bm\Sigma}
                \end{pmatrix}
                \begin{pmatrix}
                    \bm V^T \\\widetilde{\bm V}^T
                \end{pmatrix} + \widetilde{\bm Z}\bm B^T. 
            \end{aligned}
        \end{equation}
        Here, \(\widetilde{\bm Z} = \bm Z + \tau\bm W\bm B \in \R^{N_x \times 3}\) is arbitrary. The matrix \(\widetilde{\bm W} = \bm W \bm\Pi_{\bm B}^{\perp}\), with compact SVD \(\widetilde{\bm W} = \widetilde{\bm U}\widetilde{\bm \Sigma}\widetilde{\bm V}^T\), satisfies
        \begin{subequations}
            \begin{eqnarray}
                & \bm U^T\bm W = \bm 0 \implies \bm U^T\widetilde{\bm U} = \bm 0, \\
                & \bm W\bm V = \bm 0 \implies \widetilde{\bm V}^T \bm V = \bm 0, \\
                & \bm\Pi_{\bm B}^{\perp}\bm B = \bm 0 \implies \widetilde{\bm V}^T \bm B = \bm 0, \\
                & \norm{\bm W^T \bm x}_2^2
                = \norm{(\bm I - \bm B\bm B^T)\bm W^T \bm x}_2^2 + \norm{\bm B\bm B^T \bm W^T \bm x}_2^2
                \geqslant \norm{\widetilde{\bm W}^T\bm x}_2^2 \implies \norm{\widetilde{\bm W}}_2 \leqslant 1. 
            \end{eqnarray}
        \end{subequations}
        Therefore, such a decomposition exists. The equation
        \begin{equation}
            \frac{\eta}{a_{2}+\eta}\bm M \bm\Pi_{\bm B}^{\perp} 
            = \begin{pmatrix}
                \bm U & \widetilde{\bm U}
            \end{pmatrix}
            \begin{pmatrix}
                \bm\Sigma + \tau & \bm 0\\
                \bm 0 & \tau\widetilde{\bm\Sigma}
            \end{pmatrix}
            \begin{pmatrix}
                \bm V^T \\\widetilde{\bm V}^T
            \end{pmatrix}
        \end{equation}
        gives a compact SVD. Using the fact that \(\norm{\widetilde{\bm\Sigma}}_{2} = \norm{\widetilde{\bm W}}_2 \leqslant 1\), we arrive at equation \eqref{Proximal2---Implicit}.
    \end{proof}
\end{theorem}

\begin{remark}[Efficient computation of the threshold]\label{rem:bisection}
    Although equation \eqref{Proximal2---Implicit} defines \(\bm Y\) implicitly, the threshold \(\tau\) can be computed efficiently. Suppose we have computed the SVD \(\frac{\eta}{a_{2}+\eta}\bm M(\bm I - \bm B\bm B^T) = \widehat{\bm U}\widehat{\bm\Sigma}\widehat{\bm V}^T\) with \(\widehat{\bm\Sigma} = \diag\left(\sigma_1,\cdots,\sigma_n\right)\). Then \(\tau\) satisfies
    \begin{equation}\label{eq:bisection}
        0 = (a_{2}+\eta)\tau - 2\sum_{k=1}^{n} \max\B{\sigma_k - \tau, 0}.
    \end{equation}
    The right-hand side is a convex, monotonically increasing, piecewise linear function of \(\tau\) with breakpoints at \(\sigma_1, \ldots, \sigma_n\). The root can be located by binary search over the breakpoints in \(\mathcal{O}(\log n)\) comparisons, followed by an explicit solve within the linear piece, for a total cost of \(\mathcal{O}(n\log n)\) floating point operations (dominated by sorting the singular values). The SVD itself is only computed once and is the dominant computational cost. A partial SVD, computing only the singular values larger than $\tau$, could be used, but may need to be later expanded since the value of $\tau$ is not known in advance, so it is not particularly efficient.
\end{remark}

\subsection{Restarted dual FISTA for the convex problem}\label{sec:FISTA}

The Douglas--Rachford splitting works directly on the primal problem, but requires computing the proximal operator of \(g\) at each iteration, which involves a full SVD (cf.\ Theorem~\ref{thm:proximal} and Remark~\ref{rem:bisection}).   An alternative is to work on the dual problem, where the structure of the conjugate functions may lead to cheaper iterations. Since \(f\) and \(g\) are both strongly convex (with parameters \(a_1\) and \(a_2\) respectively), their convex conjugates \(f^{\ast}\) and \(g^{\ast}\) have Lipschitz continuous gradients, making the dual problem amenable to gradient-based methods.

The FISTA algorithm \cite{doi:10.1137/080716542} uses a Nesterov-type \cite{nesterov1983method} momentum method to accelerate the convergence of proximal gradient descent from \(\mathcal{O}(1/k)\) to \(\mathcal{O}(1/k^{2})\), where \(k\) is the iteration number. FISTA can be further accelerated by adaptive restart strategies \cite{O'Donoghue2015, beck2017first}. The approach of applying restarted FISTA to a dual problem for enforcing nonnegativity constraints was studied in \cite{chen2025nonneg}. We apply the same idea to the dual problem of the splitting formulation \eqref{the splitting},
\begin{equation}\label{splitting dual formulation}
    \begin{aligned}
        & \min_{\bm\Lambda} F(\bm\Lambda),  \\
        & \text{where } F(\bm\Lambda) = \B{f^{\ast}(-\bm\Lambda) + g^{\ast}(\bm\Lambda)}, 
    \end{aligned}
\end{equation}
where \(f^{\ast}\) and \(g^{\ast}\) are convex conjugates of \(f\) and \(g\) respectively, and \(\bm\Lambda \in \R^{m \times n}\) is the dual variable. Here, we require \(a_{1} > 0\) and \(a_{2} > 0\). 

Since \(f\) is \(a_{1}\)-strongly convex and \(g\) is \(a_{2}\)-strongly convex, the conjugate functions \(f^{\ast}\) and \(g^{\ast}\) are differentiable with \(\frac{1}{a_{1}}\)-Lipschitz and \(\frac{1}{a_{2}}\)-Lipschitz gradients, respectively. In particular, the dual cost function \(F(\bm\Lambda) = f^{\ast}(-\bm\Lambda) + g^{\ast}(\bm\Lambda)\) takes finite values and has Lipschitz continuous gradients.

We first write down \(f^{\ast}\):
\begin{equation}
    \begin{aligned}
        f^{\ast}(\bm\Lambda) 
        =& \sup_{\bm X \geqslant -\bm A} \left\langle \bm\Lambda, \bm X \right\rangle - \frac{a_{1}}{2} \|\bm X\|_{F}^{2} \\
        =& \left( \left\langle \bm\Lambda, \bm X \right\rangle - \frac{a_{1}}{2} \|\bm X\|_{F}^{2} \right)\bigg|_{\bm X = \max\{-\bm A, a_{1}^{-1}\bm\Lambda\}}, 
    \end{aligned}
\end{equation}
since the problem is separable, 
and using a similar procedure as the derivation of \(\mathrm{prox}_{\eta^{-1}g}\), we have
\begin{equation}
    g^{\ast}(\bm\Lambda) = \left(\left\langle \widetilde{\bm\sigma}, \bm\sigma \right\rangle - \|\bm\sigma\|_{\ell^{1}}^{2} - \frac{a_{2}}{2} \|\bm\sigma\|_{\ell^{2}}^{2}\right)\bigg|_{\bm\sigma = D_{\frac{2\|\bm\sigma\|_{\ell^{1}}}{a_{2}}}(\frac{\widetilde{\bm\sigma}}{a_{2}})},
\end{equation}
where \(\widetilde{\bm\sigma}\) is the singular value vector of \(\bm\Lambda\Pi_{\bm B}^{\perp}\) and $\|\bm\sigma\|_{\ell^{1}}$ is the sum over all the absolute values of the entries of $\bm\sigma$.

Since both \(f^{\ast}\) and \(g^{\ast}\) are smooth, the proximal-gradient (FISTA) iteration can be set up in two symmetric ways: take the gradient of one conjugate and the proximal operator of the other. We describe both and call them \emph{variant~I} (gradient of \(f^{\ast}\), proximal operator of \(g^{\ast}\)) and \emph{variant~II} (gradient of \(g^{\ast}\), proximal operator of \(f^{\ast}\)).

\paragraph{Variant I (gradient of \(f^{\ast}\), proximal operator of \(g^{\ast}\)).}
This variant requires the proximal operator of \(g^{\ast}\). Since \(g\) is \(a_2\)-strongly convex, \(g^{\ast}\) has \(\frac{1}{a_2}\)-Lipschitz gradient, and \(\mathrm{prox}_{\eta_k g^{\ast}}\) can be evaluated using the Moreau decomposition \(\mathrm{prox}_{\eta_k g^{\ast}}(\bm Z) = \bm Z - \eta_k \mathrm{prox}_{\eta_k^{-1} g}(\eta_k^{-1}\bm Z)\), where \(\mathrm{prox}_{\eta_k^{-1} g}\) is the proximal operator computed in Theorem~\ref{thm:proximal}. The iteration is
\begin{subequations}
    \begin{eqnarray}
        && \bm\Lambda_{k+1} = \mathrm{prox}_{\eta_{k} g^{\ast}}\left( \bm\Theta_{k} + \eta_{k} \nabla f^{\ast}(-\bm\Theta_{k}) \right), \\
        && t_{k+1} = \frac{1 + \sqrt{1 + 4 t_{k}^2}}{2}, \\
        && \bm\Theta_{k+1} = \bm\Lambda_{k+1} + \frac{t_{k}-1}{t_{k+1}}(\bm\Lambda_{k+1} - \bm\Lambda_{k}),
    \end{eqnarray}
\end{subequations}
where \(\bm\Lambda_{k}\) converges fast to \(\bm\Lambda\), and the initial condition is taken to be \(\bm\Theta_{0} = \bm\Lambda_{0}\) and \(t_{0} = 1\). We require \(\eta_{k} \leqslant \frac{1}{L_{f^{\ast}}} = a_{1}\) for stability.

Since this algorithm operates in the dual space, we need to recover the primal variable \(\bm X\) from \(\bm\Lambda\). We use \(\bm X = \nabla f^{\ast}(-\bm\Lambda)\), which always satisfies the nonnegativity condition \(\bm X \geqslant -\bm A\) (but not necessarily the orthogonality constraint since the iteration never \emph{exactly} reaches optimality). By the Fenchel--Young equality,
\[
    \left\langle \bm\Lambda, \bm X \right\rangle = f(\bm X) + f^{\ast}(\bm\Lambda) 
    \iff \bm X \in \partial f^{\ast}(\bm\Lambda), 
\]
the primal variable is given by
\begin{equation}\label{gradient of f*, recover primal variable}
    \bm X = \nabla f^{\ast}(-\bm\Lambda)
    = \max\{-\bm A, -a_{1}^{-1}\bm\Lambda\}.
\end{equation}

\paragraph{Variant II (gradient of \(g^{\ast}\), proximal operator of \(f^{\ast}\)).}
Here we instead take the gradient of \(g^{\ast}\) and the proximal operator of \(f^{\ast}\). The iteration is
\begin{subequations}
    \begin{eqnarray}
        && \bm\Lambda_{k+1} = \mathrm{prox}_{\eta_{k}\, [f^{\ast}(-\,\cdot\,)]}\left( \bm\Theta_{k} - \eta_{k} \nabla g^{\ast}(\bm\Theta_{k}) \right), \\
        && t_{k+1} = \tfrac{1}{2}\left(1 + \sqrt{1 + 4 t_{k}^2}\right), \quad
        \bm\Theta_{k+1} = \bm\Lambda_{k+1} + \tfrac{t_{k}-1}{t_{k+1}}(\bm\Lambda_{k+1} - \bm\Lambda_{k}),
    \end{eqnarray}
\end{subequations}
with step size \(\eta_{k} \leqslant \frac{1}{L_{g^{\ast}}} = a_{2}\). The gradient \(\nabla g^{\ast}(\bm\Lambda) = D_{\frac{2\|\bm X\|_{\ast}}{a_2}}(a_2^{-1}\bm\Lambda\bm\Pi_{\bm B}^{\perp})\) requires an SVD, while the proximal operator of \(f^{\ast}\) is now the cheap closed-form map. By the Moreau decomposition,
\begin{equation}
    \mathrm{prox}_{\eta\,[f^{\ast}(-\,\cdot\,)]}(\bm Z) = \bm Z + \eta\,\mathrm{prox}_{\eta^{-1} f}(-\eta^{-1}\bm Z) = \bm Z + \eta\max\B{-\bm A,\ -\tfrac{1}{a_1+\eta}\bm Z}.
\end{equation}
The recovered primal variable is now \(\bm X = \nabla g^{\ast}(\bm\Lambda) = D_{\frac{2\|\bm X\|_{\ast}}{a_2}}(a_2^{-1}\bm\Lambda\bm\Pi_{\bm B}^{\perp})\), which \emph{exactly} satisfies the orthogonality constraint \(\bm X\bm B = \bm 0\) and is low-rank, but is not necessarily nonnegative until convergence (the mirror image of variant~I, whose iterate is always nonnegative but only orthogonal at convergence).

\begin{remark}[Which variant to use]
Both variants converge to the same global minimizer of \eqref{F2-1 (homo)} and, as Table~\ref{tab:correction} shows, have comparable per-iteration cost (each needs one SVD per iteration) and nearly the same convergence rate, with variant~I slightly faster and cheaper (its primal recovery \eqref{gradient of f*, recover primal variable} is an elementwise clip rather than an SVD). They differ in which constraint the iterate satisfies exactly before convergence: variant~I keeps the nonnegativity \(\bm A + \bm X \geqslant \bm 0\), whereas variant~II keeps the orthogonality \(\bm X\bm B = \bm 0\). Since the entire purpose of the correction is to recover nonnegativity, we recommend variant~I: it returns a nonnegative iterate at every step and can therefore be stopped early while still yielding a physically admissible (nonnegative) solution, whereas variant~II may return a solution with small negative entries until it has fully converged.
\end{remark}

To eliminate the oscillation caused by the momentum term, restart strategies are employed. For strongly convex problems, a fixed restart with an optimal period produces linear convergence \cite{beck2017first}. Since our dual problem \eqref{splitting dual formulation} is not strongly convex, we also consider adaptive restart methods \cite{Nesterov2013, O'Donoghue2015}. The function-value-based and gradient-based restart criteria are
\begin{equation}\label{function-value restart}
    F(\bm\Lambda_{k}) > F(\bm\Lambda_{k-1}), 
\end{equation}
and 
\begin{subequations}\label{gradient restart}
    \begin{equation}
        \left\langle \bm G(\bm\Theta_{k-1}),\ \bm\Lambda_{k} - \bm\Lambda_{k-1} \right\rangle > 0,
    \end{equation}
    where 
    \begin{equation}
        \bm G(\bm\Theta_{k-1}) = \frac{1}{\eta_{k}}\left(\bm\Theta_{k-1} - \bm\Lambda_{k}\right).
    \end{equation}
\end{subequations}

\subsection{Restarted dual accelerated gradient descent for the convex problem}\label{sec:EAGD}

If both \(a_{1}\) and \(a_{2}\) are positive, the dual cost function \(F\) in \eqref{splitting dual formulation} has Lipschitz continuous gradient with Lipschitz constant \(L = \frac{1}{a_{1}} + \frac{1}{a_{2}}\). In this case, we use a fully explicit accelerated gradient descent method without proximal operators. We apply Nesterov's accelerated gradient descent \cite{nesterov1983method} to the dual problem and use the same restarting strategies as in the previous subsection.

The gradient of the dual cost function \(F\) is 
\begin{equation}
    \nabla F(\bm\Lambda) 
    = -\nabla f^{\ast}(-\bm\Lambda) + \nabla g^{\ast}(\bm\Lambda), 
\end{equation}
where \(\nabla f^{\ast}\) is given by \eqref{gradient of f*, recover primal variable} and \(\nabla g^{\ast}\) is given by 
\begin{equation}\label{gradient of g*}
    \bm X = \nabla g^{\ast}(\bm\Lambda)
    = D_{\frac{2\norm{\bm X}_{\ast}}{a_{2}}}(\frac{1}{a_{2}}\bm\Lambda \bm\Pi_{\bm B}^{\perp}).
\end{equation}
The accelerated gradient iteration is
\begin{eqnarray}
    && \bm\Lambda_{k+1} = \bm\Theta_{k} - \eta_{k} \left(\nabla g^{\ast}(\bm\Theta_{k}) - \nabla f^{\ast}(-\bm\Theta_{k})\right), \\
    && t_{k+1} = \frac{1 + \sqrt{1 + 4 t_{k}^2}}{2}, \\
    && \bm\Theta_{k+1} = \bm\Lambda_{k+1} + \frac{t_{k}-1}{t_{k+1}}(\bm\Lambda_{k+1} - \bm\Lambda_{k}),
\end{eqnarray}
where \(\bm\Lambda_{k}\) converges to \(\bm\Lambda\), and the initial condition is \(\bm\Theta_{0} = \bm\Lambda_{0}\) and \(t_{0} = 1\). We require \(\eta_{k} \leqslant \frac{1}{L_{F}} = \frac{a_{1} a_{2}}{a_{1} + a_{2}}\) for stability. Similar to the FISTA, our restarting strategies include fixed restart, function-value-based restart \eqref{function-value restart}, and gradient-based restart \eqref{gradient restart}.

\subsection{Dual PR+ conjugate gradient for the convex problem}\label{sec:CG}

We next apply the nonlinear conjugate gradient method to the dual problem \eqref{splitting dual formulation}. We use the Polak--Ribi\`{e}re plus (PR+) variant \cite{doi:10.1137/1028154} with an explicit step size formula from \cite{Sun2001} to avoid line search. The algorithm is
\begin{eqnarray}
    && \bm G_{k} = \nabla F(\bm\Lambda_{k}), \\
    && \beta_{k} = \begin{cases}
        0, & k = 1, \\
        \frac{\left\langle \bm G_{k}, \bm G_{k} - \bm G_{k-1} \right\rangle}{\norm{\bm G_{k-1}}_{F}^{2}}, & k > 1, 
    \end{cases} \\
    && \bm D_{k} = \beta_{k} \bm D_{k-1} - \bm G_{k}, \\
    && \alpha_{k} = -\frac{\delta \left\langle \bm G_{k}, \bm D_{k} \right\rangle}{\norm{\bm D_{k}}_{F}^{2}} \\
    && \bm\Lambda_{k+1} = \bm\Lambda_{k} + \alpha_{k}\bm D_{k}. 
\end{eqnarray}
Here, according to \cite{Sun2001}, we should require \(\delta \in (0, \frac{1}{L_{F}})\) for stability.

\subsection{Dual L-BFGS for the convex problem}\label{sec:LBFGS}

The L-BFGS algorithm \cite{liu1989limited} is a quasi-Newton method that uses a limited number of past iterates to approximate the inverse Hessian, requiring only gradient evaluations. Applied to the dual problem \eqref{splitting dual formulation}, the algorithm is
\begin{eqnarray}
    && \bm G_{k} = \nabla F(\bm\Lambda_{k}), \\
    && \bm D_{k} = -\bm H_{k} \bm G_{k}, \\
    && \alpha_{k} = \arg\min_{\alpha \geqslant 0} F(\bm\Lambda_{k} + \alpha \bm D_{k}), \\
    && \bm\Lambda_{k+1} = \bm\Lambda_{k} + \alpha_{k} \bm D_{k}, \\
    && \begin{aligned}
        \bm H_{k+1} = & \left(\prod_{j=k-p+1}^{k} \bm J_{j}\right)^{T} \bm H_{0}^{k} \prod_{j=k-p+1}^{k} \bm J_{j} \\& + \sum_{i=1}^{p} \gamma_{k-p+i} \left(\prod_{j=k-p+i+1}^{k} \bm J_{j}\right)^{T} \bm s_{k-p+i} \bm s_{k-p+i}^{T} \prod_{j=k-p+i+1}^{k} \bm J_{j}.
    \end{aligned}
\end{eqnarray}
Here, the iterates are assumed to be vectors, \(\bm J_{k} = \bm I - \gamma_{k}\bm y_{k}\bm s_{k}^{T}\), \(\bm y_{j} = \bm G_{j+1} - \bm G_{j}\), \(\bm s_{j} = \bm\Lambda_{j+1} - \bm\Lambda_{j}\), \(\gamma_{j} = \frac{1}{\bm y_{j}^{T} \bm s_{j}}\), and \(p\) is the number of history iterates to be stored. The matrix-vector product of \(\bm H_{k}\) can be efficiently implemented through two iterations presented in \cite{liu1989limited}. The exact line search can be replaced by a Wolfe-type inexact search.

\subsection{Tangent-space accelerated alternating projection for the non-convex problem}\label{sec:TAP}

For the rank-constrained formulation \eqref{F2-2}, we develop an alternating projection algorithm that uses the manifold structure to avoid full SVD computations. We first introduce the relevant manifold and its projection properties.

\begin{definition}
    Let \(M_{r, \bm B} = \B{\bm X\in\R^{m\times n} \big| \rank(\bm X) \leqslant r,\ \bm X\bm B = \bm 0}\) denote the set of rank-at-most-\(r\) matrices satisfying the orthogonality constraint.
\end{definition}

\begin{lemma}\label{lem:proj-MrB}
    The Frobenius-norm projection onto \(M_{r,\bm B}\) is
    \begin{equation}
        P_{M_{r,\bm B}}(\bm Y)
        = P_{M_{r}}\left(\bm Y\bm\Pi_{\bm B}^{\perp}\right), 
    \end{equation}
    where \(P_{M_{r}}(\cdot)\) is the standard rank-\(r\) SVD truncation. 
    
    \begin{proof}
        Let \(\bm Y\bm\Pi_{\bm B}^{\perp} = \bm U\bm\Sigma\bm V^T\) be a full SVD. For any \(\bm Z \in M_{r,\bm B}\), 
        \begin{equation}
            \begin{aligned}
                \norm{\bm Y - \bm Z}_{F}^2
                &= \norm{\bm Y\bm\Pi_{\bm B}}_{F}^2 + \norm{\bm Y\bm\Pi_{\bm B}^{\perp} - \bm Z}_{F}^2 \\
                &= \norm{\bm Y\bm\Pi_{\bm B}}_{F}^2 + \norm{\bm Y\bm\Pi_{\bm B}^{\perp}}_{F}^2 - 2\tr\left(\bm Z^T\bm Y\bm\Pi_{\bm B}^{\perp}\right) + \norm{\bm Z}_{F}^2 \\
                &= \norm{\bm Y\bm\Pi_{\bm B}}_{F}^2 + \norm{\bm\Sigma - \widehat{\bm Z}}_{F}^2 \\
                &\geqslant \norm{\bm Y\bm\Pi_{\bm B}}_{F}^2 +
                \sum_{k = r+1}^{n} \Sigma_{k,k}^2,
            \end{aligned}
        \end{equation}
        where \(\widehat{\bm Z} = \bm U^T\bm Z\bm V\) has rank at most \(r\). The inequality is the Eckart--Young--Mirsky theorem: since \(\Sigma_{k,k}\) are the singular values of \(\bm\Sigma\), the best rank-\(r\) Frobenius approximation retains the largest \(r\) of them, with equality if and only if \(\widehat{\bm Z} = P_{M_r}(\bm\Sigma)\), i.e. \(\bm Z = P_{M_{r}}\left(\bm Y\bm\Pi_{\bm B}^{\perp}\right)\). It is straightforward to check that \(\bm Z = P_{M_{r}}\left(\bm Y\bm\Pi_{\bm B}^{\perp}\right) \in M_{r,\bm B}\). 
    \end{proof}
\end{lemma}

We first consider the alternating projection algorithm that projects onto the nonnegative set and onto \(M_{r,\bm B}\), as shown in Algorithm~\ref{AP for non-convex problem}.

\begin{algorithm}[htbp]
    \SetAlgoLined

    \KwIn{\(\bm X_0\), \(\bm A\), \(\bm B\), \(r\), \(\text{maxit}\)}

    \KwOut{\(\bm X_{\text{maxit}+1}\)}

    \For{\(k=1:\text{maxit}\)}{

        \(\bm Y_{k} = \pi_{+}\left(\bm X_k\right)\) \tcp*{element-wise}

        \(\bm X_{k+1} = P_{M_{r,\bm B}}\left(\bm Y_{k}\right)\) \tcp*{requires SVD}
    }

    \caption{Alternating projection (AP) for the non-convex problem}
    \label{AP for non-convex problem}
\end{algorithm}

The expensive step is the projection \(P_{M_{r,\bm B}}\), which requires a full SVD. Following \cite{song2020tangent}, we replace this projection by first projecting onto the tangent space of \(M_{r,\bm B}\) at the current iterate \(\bm X_k\) and then truncating to rank \(r\). This reduces the SVD cost from a full \(m \times n\) SVD to a \(2r \times 2r\) SVD. We now introduce the tangent-space machinery. In the following, we assume that the iterates \(\bm X_k\) have exact rank \(r\), so that the tangent space of the fixed-rank manifold is well-defined. This assumption holds generically in practice, especially since we choose $r$ to be small.

\begin{lemma}
    The tangent space of the fixed-rank manifold \(M_{r}\) at \(\bm X = \bm U\bm\Sigma\bm V^T \in M_{r}\) (compact SVD form, with \(\bm\Sigma\) having all positive diagonal entries) is
    \begin{equation}
        \begin{aligned}
            & T_{M_{r}}(\bm X) \\
            &= \B{\bm U\bm W^T + \bm Z \bm V^T \big|\bm W\in\R^{n\times r}, \bm Z \in\R^{m\times r}} \\
            &= \B{\bm U\bm M\bm V^T + \bm U_{p}\bm V^T + \bm U\bm V_{p}^T \big| \bm U_{p}^T\bm U = \bm 0, \bm V_{p}^T\bm V = \bm 0, \bm M\in\R^{r\times r}, \bm U_{p}\in\R^{m\times r}, \bm V_{p}\in\R^{n\times r}} \\
            &= \B{\begin{pmatrix}
                \bm U & \bm U^{\perp}
            \end{pmatrix}
            \begin{pmatrix}
                \bm C_1 & \bm C_2 \\
                \bm C_3 & \bm 0
            \end{pmatrix}
            \begin{pmatrix}
                \bm V & \bm V^{\perp}
            \end{pmatrix}^T \bigg| \bm C_1 \in\R^{r\times r}, \bm C_2\in\R^{r\times(n-r)}, \bm C_3\in\R^{(m-r)\times r}}. 
        \end{aligned}
    \end{equation}

    \begin{proof}
        See \cite{zheng2022riemannian} and \cite{vandereycken2013low}.
    \end{proof}
\end{lemma}

\begin{lemma}\label{lemma:projTmr}
    The least square projection (in Frobenius norm) onto \(T_{M_r}(\bm X)\) is 
    \begin{equation}
        P_{T_{M_r}(\bm X)}(\bm Y)
        = \bm\Pi_{\bm U}\bm Y + \bm Y \bm\Pi_{\bm V} - \bm\Pi_{\bm U}\bm Y\bm\Pi_{\bm V}, 
    \end{equation}
    where \(\bm X = \bm U\bm\Sigma\bm V^T\) is a rank-\(r\) SVD.
    
    \begin{proof}
        See \cite{zheng2022riemannian} and \cite{vandereycken2013low}.
    \end{proof}
\end{lemma}
Note that these computations can be done efficiently if the matrix multiplication is done in the right order; e.g., $\bm\Pi_{\bm U}\bm Y = \bm U\bm U^T \bm Y$ is calculated as $\bm U \left( \bm U^T \bm Y\right)$, and the final outer multiplication is not done explicitly but left in factored form (see Remark~\ref{rmk:factored}).

\begin{lemma}
    The tangent space of \(M_{r, \bm B}\) at \(\bm X = \bm U\bm\Sigma\bm V^T \in M_{r, \bm B}\) (compact SVD form) is 
    \begin{equation}
        \begin{aligned}
            & T_{M_{r,\bm B}}(\bm X) \\
            &= \B{\bm U\bm W^T + \bm Z \bm V^T \big|\bm W\in\R^{n\times r}, \bm Z \in\R^{m\times r}, \bm W^T\bm B = \bm 0} \\
            &= \B{\bm U\bm M\bm V^T + \bm U_{p}\bm V^T + \bm U\bm V_{p}^T \bm\Pi_{\bm B}^{\perp} \big| \bm U_{p}^T\bm U = \bm 0, \bm V_{p}^T\bm V = \bm 0, \bm M\in\R^{r\times r}, \bm U_{p}\in\R^{m\times r}, \bm V_{p}\in\R^{n\times r}}. 
        \end{aligned}
    \end{equation}

    \begin{proof}
        Since \(\bm X \in M_{r,\bm B}\), we have \(\bm V^T\bm B = \bm 0\). The orthogonality constraint \(\bm W^T\bm B = \bm 0\) follows from differentiating \(\bm X\bm B = \bm 0\) along the manifold.
    \end{proof}
\end{lemma}

\begin{lemma}
    The least square projection (in Frobenius norm) onto \(T_{M_r,\bm B}(\bm X)\) is 
    \begin{equation}
        P_{T_{M_{r,\bm B}}(\bm X)}(\bm Y)
        = P_{T_{M_r}(\bm X)}(\bm Y)\bm\Pi_{\bm B}^{\perp}
        = P_{T_{M_r}(\bm X)}\left(\bm Y\bm\Pi_{\bm B}^{\perp}\right). 
    \end{equation}

    \begin{proof}
        Since \(\bm X \in M_{r,\bm B}\), we have \(\bm V^T\bm B = \bm 0\). As a result, there exists a matrix \(\bm K\) such that the matrix \(\begin{pmatrix}
            \bm V & \bm B & \bm K
        \end{pmatrix}\) is orthonormal. Using matrix \(\bm K\), we can rewrite the tangent space as 
        \begin{equation}
            \begin{aligned}
                &T_{M_{r,\bm B}}(\bm X) \\
                &= \B{\begin{pmatrix}
                    \bm U & \bm U^{\perp}
                \end{pmatrix}
                \begin{pmatrix}
                    \bm C_1 & \bm 0 & \bm C_2 \\
                    \bm C_3 & \bm 0 & \bm 0
                \end{pmatrix}
                \begin{pmatrix}
                    \bm V & \bm B & \bm K
                \end{pmatrix}^T \bigg| \bm C_1\in\R^{r\times r}, \bm C_2 \in\R^{r\times(n-r-3)}, \bm C_3 \in\R^{(m-r)\times r}}. 
            \end{aligned}
        \end{equation}
        Then, the least square projection is given by 
        \begin{equation}
            \begin{aligned}
                P_{T_{M_{r,\bm B}}(\bm X)}(\bm Y)
                &= \bm\Pi_{\bm U}\bm Y\bm\Pi_{\bm V} + \bm\Pi_{\bm U}^{\perp} \bm Y \bm\Pi_{\bm V} + \bm\Pi_{\bm U} \bm Y \bm\Pi_{\bm K} \\
                &= \bm\Pi_{\bm U} \bm Y \bm\Pi_{\bm B}^{\perp} + \bm\Pi_{\bm U}^{\perp} \bm Y \bm\Pi_{\bm V} \\
                &= \bm\Pi_{\bm U} \bm Y \bm\Pi_{\bm B}^{\perp} + \bm\Pi_{\bm U}^{\perp} \bm Y \bm\Pi_{\bm V}\bm\Pi_{\bm B}^{\perp} = P_{T_{M_r}(\bm X)}(\bm Y)\bm\Pi_{\bm B}^{\perp},\\
                &= \bm\Pi_{\bm U} \bm Y \bm\Pi_{\bm B}^{\perp} + \bm\Pi_{\bm U}^{\perp} \bm Y \bm\Pi_{\bm B}^{\perp} \bm\Pi_{\bm V} = P_{T_{M_r}(\bm X)}\left(\bm Y\bm\Pi_{\bm B}^{\perp}\right).
            \end{aligned}
        \end{equation}
    \end{proof}
\end{lemma}

\begin{remark}\label{rmk:factored}
    Since \(\rank(P_{T_{M_{r,\bm B}}(\bm X)}(\bm Y)) \leqslant 2r\), the tangent-space projection can be stored in a low-rank factored form.
\end{remark}

Using these results, we obtain the tangent-space accelerated alternating projection (TAP) algorithm in Algorithm~\ref{TAP for non-convex problem}. At each iteration, we project \(\bm Y_{k}\) onto the tangent space of \(M_{r,\bm B}\) at \(\bm X_{k}\), and then truncate to rank \(r\). The implementation of this combined step is given in Algorithm~\ref{TP for non-convex problem}.

\begin{algorithm}[htbp]
    \SetAlgoLined

    \KwIn{\(\bm X_0\), \(\bm A\), \(\bm B\), \(r\), \(\text{maxit}\)}

    \KwOut{\(\bm X_{\text{maxit}+1}\)}

    \For{\(k=1:\text{maxit}\)}{

        \(\bm Y_{k} = \pi_{+}\left(\bm X_k\right)\) \;

        \(\bm X_{k+1} = P_{M_r}\left(P_{T_{M_{r,\bm B}}(\bm X_{k})}\left(\bm Y_{k}\right)\right)\)\tcp*{rank $r$, and SVD also computed \& stored}
    }

    \caption{TAP for non-convex problem}
    \label{TAP for non-convex problem}
\end{algorithm}

\begin{algorithm}[htbp]
    \SetAlgoLined

    \KwIn{\(\bm X_k = \bm U\bm\Sigma\bm V^T\) (rank-\(r\) compact SVD), \(\bm B\), \(\bm Y_{k}\), \(r\)}

    \KwOut{\(\bm X_{k+1}\) (in compact SVD form)}

    \(\bm Q_1 \bm R_1 = \bm\Pi_{\bm U}^{\perp}\bm Y_k\bm V\) \tcp*{compact QR}

    \(\bm Q_2 \bm R_2 = \bm\Pi_{\bm V}^{\perp}\bm Y_k^T\bm U\) \tcp*{compact QR \\by orthogonality: \(\bm U^T\bm Q_1 = \bm 0\) and \(\bm V^T \bm Q_2 = \bm 0\)}

    \(\widetilde{\bm Q} \widetilde{\bm R} = \bm\Pi_{\bm B}^{\perp}\begin{pmatrix}
        \bm V & \bm Q_2
    \end{pmatrix}\) \tcp*{compact QR}

    \(\bm E = \begin{pmatrix}
        \bm U^T \bm Y_k\bm V & \bm R_2 \\
        \bm R_1 & \bm 0
    \end{pmatrix} \widetilde{\bm R}^T\) \tcp*{\(\bm E\) is \(2r\times 2r\)}

    \(\widehat{\bm U}\widehat{\bm\Sigma}\widehat{\bm V}^T = \bm E\) \tcp*{\(2r\times 2r\) SVD \\
    rank-\(2r\) factorization: \(P_{T_{M_{r,\bm B}}(\bm X_{k})}(\bm Y_{k}) = \begin{pmatrix}
        \bm U & \bm Q_1
    \end{pmatrix}\widehat{\bm U}\,\widehat{\bm\Sigma}\,\widehat{\bm V}^T\widetilde{\bm Q}^T\)}

    Truncate to rank \(r\): keep only the top \(r\) singular triplets of \(\widehat{\bm U}\widehat{\bm\Sigma}\widehat{\bm V}^T\) to form \(\bm X_{k+1}\)

    \caption{TP operator \(P_{M_r}\left(P_{T_{M_{r,\bm B}}(\bm X)}\left(\bm Y\right)\right)\) for non-convex problem}
    \label{TP for non-convex problem}
\end{algorithm}

\begin{remark}[Computational cost]
    All iterates \(\bm X_{k}\) are stored in rank-\(r\) factored form \(\bm U\bm\Sigma\bm V^T\). The computational cost per iteration of Algorithm~\ref{TP for non-convex problem} is \(\mathcal{O}(mnr + (m+n)r^2 + r^3)\), which is significantly less than the \(\mathcal{O}(mn\min(m,n))\) cost of the full SVD required by the algorithms for the convex formulation.
\end{remark}

\section{Numerical experiments}\label{sec:numerics}

\subsection{Test problem and setup}

We test all algorithms on a 1D1V Landau damping problem discretized on an \(N_x = 64\), \(N_v = 128\) grid. The low-rank solution has \(\rank(\bm A_1) = 3\) and \(\rank(\bm A_2) = 26\). The input data are shown in Figure~\ref{1D1V data}. The scaled matrix \(\bm A\) has a negativity measure of \(\norm{\min\{\bm A, \bm 0\}}_F / \norm{\bm A}_F = 3.08\%\), which provides a lower bound on the relative correction error: any nonnegative correction must have \(\norm{\bm X}_F / \norm{\bm A}_F \geqslant 3.08\%\).

\begin{figure}[htbp]
    \begin{center}
        \includegraphics[width=\textwidth]{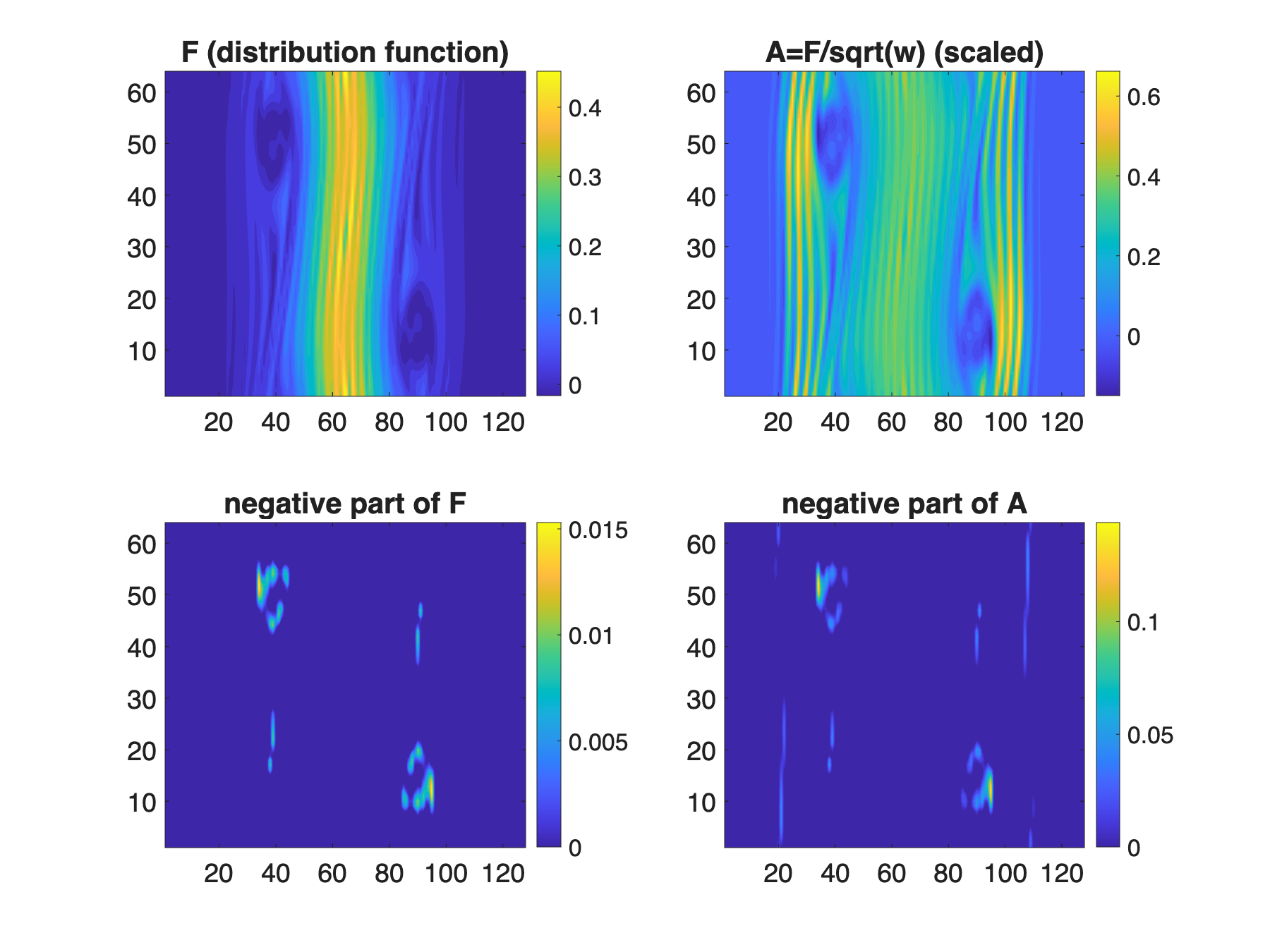}
        \caption{Input matrix \(\bm F\), scaled matrix \(\bm A\), and their negative parts. }
        \label{1D1V data}
    \end{center}
\end{figure}

For all algorithms applied to the convex formulation \eqref{F2-1 (homo)}, after experimentation (see Fig.~\ref{fig:convex-1}) we find that $a=1000$ is a good compromise between the two objectives, and henceforth set $a=1000$.
We also set \(a_1 = a_2 = a/2\), where \(a_1\) and \(a_2\) are the splitting parameters of the Douglas--Rachford decomposition \eqref{the splitting}, and note that the choice of splitting \(a = a_1 + a_2\) does not produce visible differences in the results. Each convergence plot shows the relative correction \(\norm{\bm X}_F / \norm{\bm A}_F\), the orthogonality violation \(\norm{\bm X\bm B}_F\), the rank of \(\bm X\), and the relative change between successive iterates, for several values of the parameter \(a\) or the rank constraint \(r\).

The converged relative correction \(\norm{\bm X}_F / \norm{\bm A}_F\) and the rank of \(\bm X\) depend on the penalty parameter \(a\). Computed using the Douglas--Rachford splitting with 1000 iterations, as \(a\) increases through \(1, 100, 1000, 10^4, 10^5\) the relative correction is \(4.03\%, 3.31\%, 3.29\%, 3.29\%, 3.29\%\), decreasing monotonically toward the theoretical lower bound of \(3.08\%\), while the rank of \(\bm X\) is \(38, 23, 33, 39, 45\); this reflects the trade-off between correction quality and low-rank structure controlled by \(a\). In the comparison below we fix \(a = 1000\).

Table~\ref{tab:correction} compares all algorithms on the \(64\times 128\) data, reporting the relative correction \(\norm{\bm X}_F/\norm{\bm A}_F\), the rank of \(\bm X\), the orthogonality violation \(\norm{\bm X\bm B}_F\), the minimum of the corrected solution \(\min(\bm A+\bm X)\) (nonnegativity), the number of iterations to convergence, and the total CPU time. 
Each iterative method is run until the relative change of the iterate falls below \(10^{-8}\) or a cap of \(1000\) iterations is reached. TAP uses rank \(r = 20\). For reference we also include a naive low-rank baseline: take \(\bm X = \max\{-\bm A,\bm 0\}\) (which makes \(\bm A+\bm X\) nonnegative), truncate it to rank \(r\) by an SVD, and project the result onto \(\bm\Pi_{\bm B}^{\perp}\) to enforce \(\bm X\bm B = \bm 0\). This is a direct, non-iterative construction. The baseline is essentially free and satisfies the rank and orthogonality constraints exactly, but it does not enforce nonnegativity, so \(\min(\bm A+\bm X) = -1.7\times 10^{-2} < 0\). All convex algorithms converge to the same global minimizer (relative correction \(3.29\%\), rank \(33\)). The methods that recover the primal iterate as \(\bm X = \nabla f^*(-\bm\Lambda)\) (DR, dual FISTA~I, dual AGD, dual PR+ CG, and dual L-BFGS) keep \(\min(\bm A+\bm X) = 0\) exactly at every iterate and satisfy the orthogonality constraint only in the limit. Conversely, dual FISTA~II recovers \(\bm X = \nabla g^*(\bm\Lambda)\), which is exactly orthogonal and low rank (\(\norm{\bm X\bm B}_F\) at machine precision) but only nonnegative in the limit (see Section~\ref{sec:FISTA}). The two FISTA variants converge at nearly the same rate, with variant~I slightly faster and cheaper per iteration. Douglas--Rachford splitting reaches the minimizer in the fewest iterations, whereas the dual PR+ conjugate gradient is the slowest and does not meet the tolerance within \(1000\) iterations (its iterate has not collapsed to rank \(33\) and its orthogonality violation remains \(6\times 10^{-9}\)). Although L-BFGS converges in relatively few iterations, its line search evaluates the objective several times per iteration, so its total time is larger. TAP attains a comparable correction with exact orthogonality and machine-precision nonnegativity at a lower rank and the lowest cost among the iterative methods.

\begin{table}[htbp]
\begin{center}

\begin{tabular}{lcccccc}
\toprule
method & \(\norm{\bm X}_F / \norm{\bm A}_F\) & \(\rank(\bm X)\) & \(\norm{\bm X\bm B}_F\) & \(\min(\bm A+\bm X)\) & iters & time (s) \\
\midrule
DR splitting  & \(3.29\%\) & \(33\) & \(7.4\times10^{-10}\) & \(0\)                 & \(65\)   & \(0.10\) \\
dual FISTA I  & \(3.29\%\) & \(33\) & \(3.2\times10^{-12}\) & \(0\)                 & \(201\)  & \(0.28\) \\
dual FISTA II & \(3.29\%\) & \(33\) & \(1.4\times10^{-16}\) & \(-1.9\times10^{-16}\) & \(214\)  & \(0.39\) \\
dual AGD      & \(3.29\%\) & \(33\) & \(3.8\times10^{-12}\) & \(0\)                 & \(308\)  & \(0.33\) \\
dual PR+ CG   & \(3.29\%\) & \(47\) & \(6.4\times10^{-9}\)  & \(0\)                 & \(1000\) & \(1.08\) \\
dual L-BFGS   & \(3.29\%\) & \(33\) & \(5.6\times10^{-10}\) & \(0\)                 & \(150\)  & \(0.55\) \\
TAP           & \(3.31\%\) & \(20\) & \(5.2\times10^{-17}\) & \(-1.3\times10^{-16}\) & \(45\)   & \(0.04\) \\
baseline      & \(2.93\%\) & \(20\) & \(7.4\times10^{-17}\) & \(-1.7\times10^{-2}\)  & \(1\)    & \(0.001\) \\
\bottomrule
\end{tabular}

\caption{Comparison of all algorithms on the \(64\times 128\) Landau data: relative correction, rank, orthogonality violation \(\norm{\bm X\bm B}_F\), nonnegativity \(\min(\bm A+\bm X)\), iterations to convergence, and total CPU time. Convex methods use \(a=1000\), the DR step size \(\eta=7a\) and relaxation \(\alpha=0.88\), and TAP uses \(r=20\). Convergence is declared when the relative change of the iterate drops below \(10^{-8}\), capped at \(1000\) iterations, and the baseline is a direct one-shot construction (its ``iterations'' entry is \(1\)). The convex methods reach the same minimizer. All of them except dual FISTA~II enforce nonnegativity exactly (\(\min(\bm A+\bm X)=0\) at every iterate), whereas dual FISTA~II instead enforces orthogonality exactly (\(\bm X\bm B=\bm 0\)), as discussed in Section~\ref{sec:FISTA}. PR+ conjugate gradient does not meet the tolerance within \(1000\) iterations. TAP and the baseline satisfy \(\bm X\bm B=\bm 0\) to machine precision by construction. The baseline is the fastest but violates nonnegativity, \(\min(\bm A+\bm X)<0\). CPU times are indicative and machine-dependent.}
\label{tab:correction}
\end{center}
\end{table}

\subsection{Results for the convex formulation}

\paragraph{Douglas--Rachford splitting.} We use step size \(\eta = 7a\) and relaxation parameter \(\alpha = 0.88\) (jointly tuned for the fastest convergence at \(a = 1000\), Figure~\ref{fig:convex-2}) in Algorithm~\ref{DR for convex problem}, starting from \(\bm Z_1 = \bm 0\) and running 1000 iterations. The convergence curves for different values of \(a\) are shown in Figure~\ref{fig:convex-1}(a).

\paragraph{Dual methods.} The restarted dual FISTA, restarted dual accelerated gradient descent, PR+ conjugate gradient, and L-BFGS all have comparable per-iteration cost to the Douglas--Rachford splitting. For the FISTA and accelerated gradient methods, we use step size \(\eta_k = a_1\) and test three restarting strategies: fixed restart (period 100), function-value-based restart, and gradient-based restart. Figure~\ref{fig:convex-1}(b) shows the fixed-restart FISTA and compares its two dual variants (Section~\ref{sec:FISTA}), and Figure~\ref{fig:convex-2} compares the convergence of all nine convex algorithms at the representative value \(a = 1000\).

\begin{figure}[htbp]
    \begin{center}
        \includegraphics[width=\textwidth]{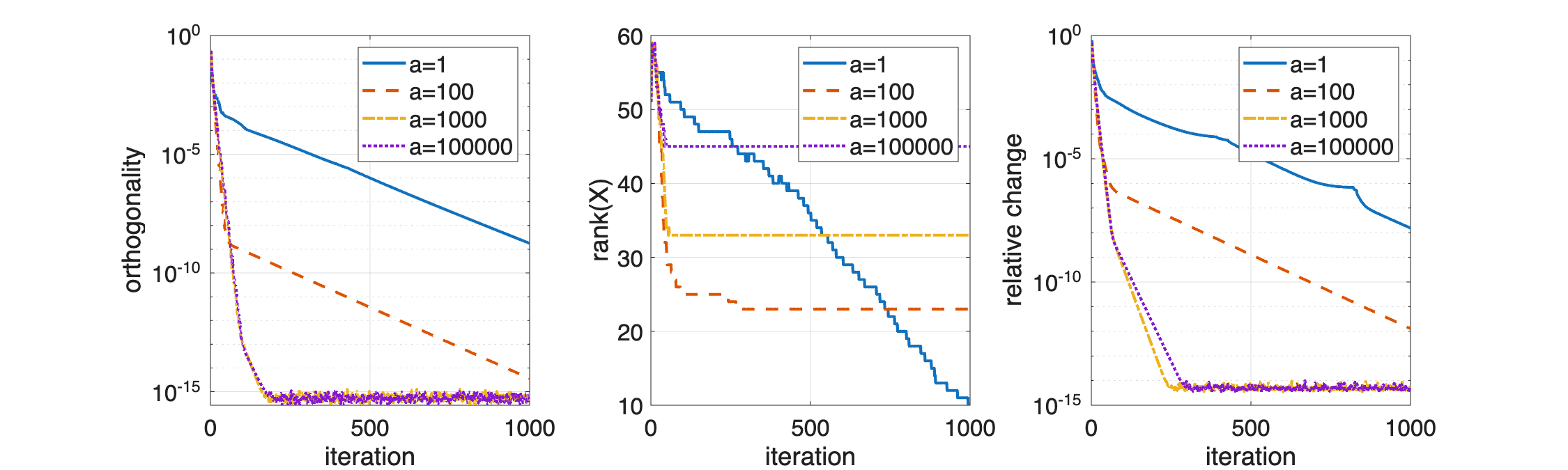}\\[2pt]
        {\small (a) Douglas--Rachford splitting} \\[6pt]
        \includegraphics[width=\textwidth]{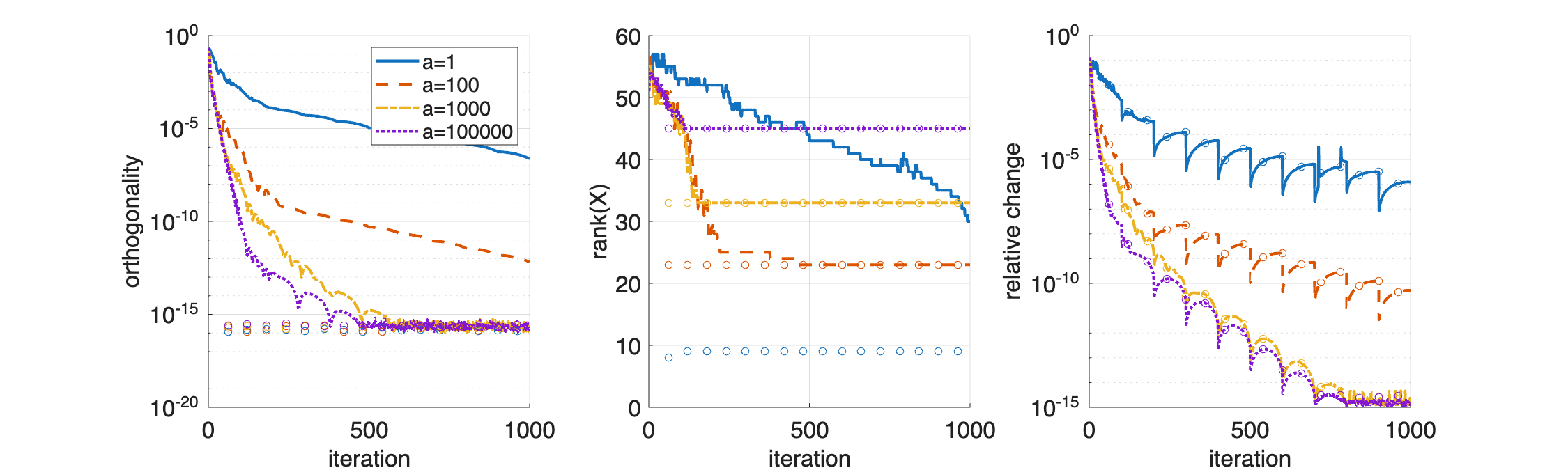}\\[2pt]
        {\small (b) Dual FISTA with fixed restart, the two variants overlaid}
        \caption{Convergence of two convex algorithms for different values of \(a\). (a) Douglas--Rachford splitting. (b) Dual FISTA with fixed restart, overlaying Variant~I (lines, gradient of \(f^*\) and proximal operator of \(g^*\)) and Variant~II (circles, gradient of \(g^*\) and proximal operator of \(f^*\)), with curves of the same color corresponding to the same value of \(a\). Each row shows the orthogonality violation \(\norm{\bm X\bm B}_F\), the rank of \(\bm X\), and the relative change between iterates. Variant~II keeps the orthogonality violation at machine precision throughout, since its iterate satisfies \(\bm X\bm B = \bm 0\) exactly.}
        \label{fig:convex-1}
    \end{center}
\end{figure}

\begin{figure}[htbp]
    \begin{center}
        \includegraphics[width=\textwidth]{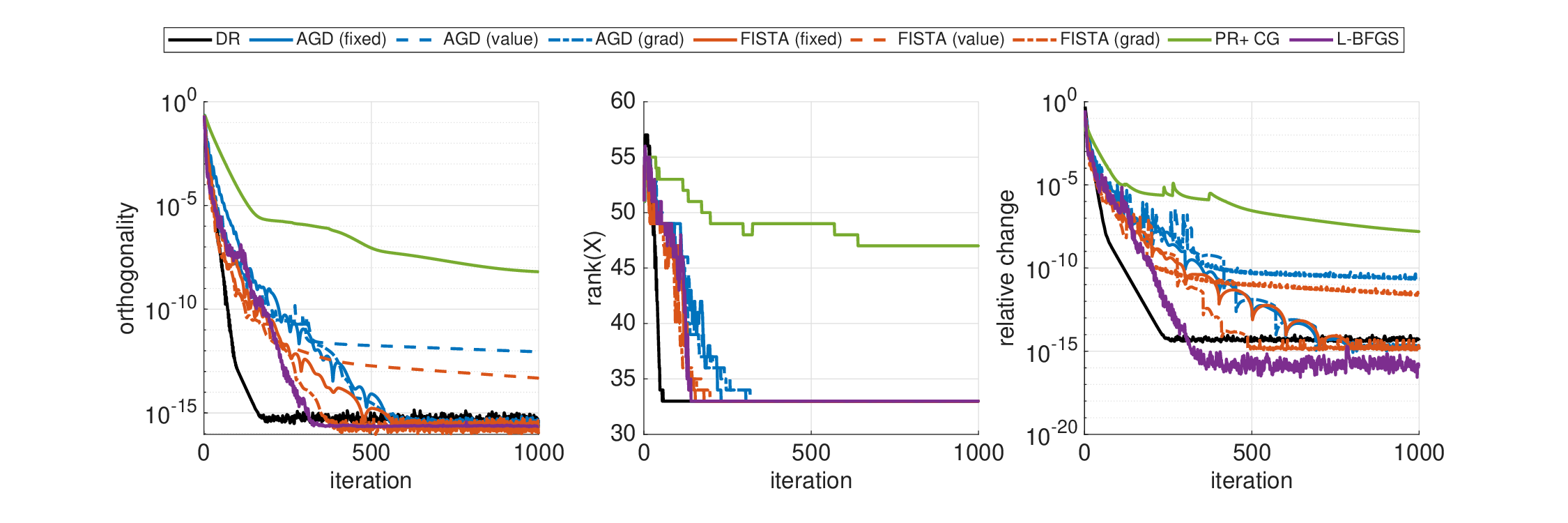}
        \caption{Convergence of all nine convex algorithms at the single penalty \(a = 1000\) (\(a_1 = a_2 = a/2\)): the orthogonality violation \(\norm{\bm X\bm B}_F\), the rank of \(\bm X\), and the relative change between iterates versus iteration. Color encodes the method family (Douglas--Rachford, accelerated gradient, FISTA, PR+ conjugate gradient, L-BFGS) and line style encodes the restart strategy (fixed, function-value, gradient), with dual FISTA shown for Variant~I. The DR step size \(\eta = 7a\) and relaxation \(\alpha = 0.88\), jointly tuned at this \(a\), give the fastest convergence (\(65\) iterations, versus \(114\) at the untuned \(\eta = 10^{0.5}a\), \(\alpha = 0.9\)). All methods reach the same minimizer (rank \(33\), relative correction \(3.29\%\)) except the PR+ conjugate gradient, which has not converged within 1000 iterations.}
        \label{fig:convex-2}
    \end{center}
\end{figure}

\subsection{Results for the non-convex formulation}

For the tangent-space accelerated alternating projection (TAP, Algorithm~\ref{TAP for non-convex problem}), we test different rank constraints \(r\) and run 200 iterations starting from \(\bm X_0 = -\min\{\bm A, \bm 0\}\). The convergence curves are shown in Figure~\ref{1D1V TAP}.

\begin{figure}[htbp]
    \begin{center}
        \includegraphics[width=\textwidth]{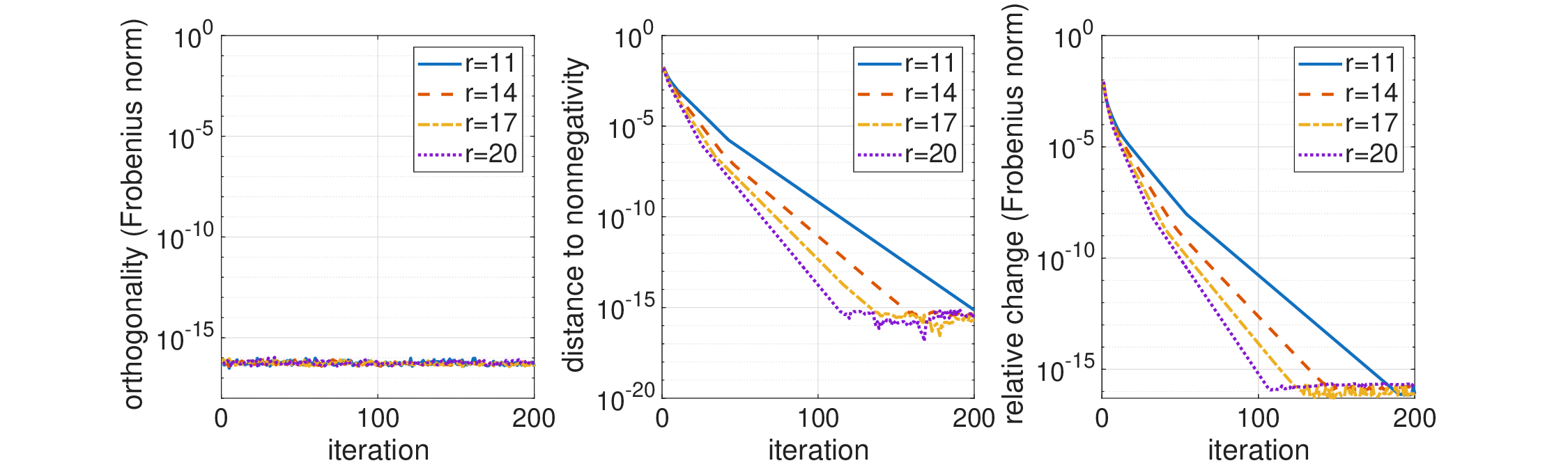}
        \caption{Tangent-space accelerated alternating projection (TAP): convergence for different rank constraints \(r\). Left: orthogonality violation \(\norm{\bm X\bm B}_F\), which stays at machine precision throughout because TAP keeps \(\bm X\bm B = \bm 0\) exactly at every iterate. Middle: distance to the nonnegative set, \(\max(-\min\{\bm A+\bm X,\bm 0\})\). Right: relative change between successive iterates.}
        \label{1D1V TAP}
    \end{center}
\end{figure}

\subsection{\texorpdfstring{Empirical cost scaling with problem size}{Empirical cost scaling with problem size}}\label{sec:scaling}

The timing experiments above use a single $64\times 128$ matrix. To verify the cost scaling empirically, we run the conservative low-rank scheme of \cite{guo2022conservative} for a strong Landau damping problem on a sequence of refined phase-space grids and measure the cost as the matrix size grows. We use the strong Landau damping initial condition $f_0(x,v) = \left(1 + \tfrac12\cos\tfrac{x}{2}\right)\frac{1}{\sqrt{2\pi}}e^{-v^2/2}$ on $\Omega_x\times\Omega_v = [0,4\pi]\times[-6,6]$ and take the low-rank solution snapshot at $t = 10$, by which time the SVD-type truncation has introduced negative entries. For each grid we form the rescaled matrix $\bm A$ and its conservative orthogonal split $\bm A = \bm A_1 + \bm A_2$ with $\bm A_1 = \bm A\bm B\bm B^T$, exactly as in Sections~\ref{sec:rescaled}--\ref{sec:correction-formulation}. We use four grids with $N_v = 2N_x$, listed in Table~\ref{tab:scaling}. The orthogonality $\bm A_2\bm B = \bm 0$ holds to machine precision ($\norm{\bm A_2\bm B}_F \lesssim 10^{-13}$) on every grid. We emphasize that this strong Landau damping data (perturbation amplitude $\tfrac12$) is different from the data used in the rest of the paper (Figure~\ref{1D1V data} and Table~\ref{tab:correction}). We use it here because the same low-rank solver then generates a self-consistent family of nonnegative-correction problems across all grid resolutions, which is what we need to measure the cost scaling. Figure~\ref{fig:data_sizes} shows the low-rank solution and its negative part for the $64\times 128$ and $256\times 512$ grids: the solution exhibits the same filamentary structure on both grids, while the truncation-induced negative entries become finer and smaller in magnitude as the grid is refined.

\begin{figure}[htbp]
    \begin{center}
        \includegraphics[width=\textwidth]{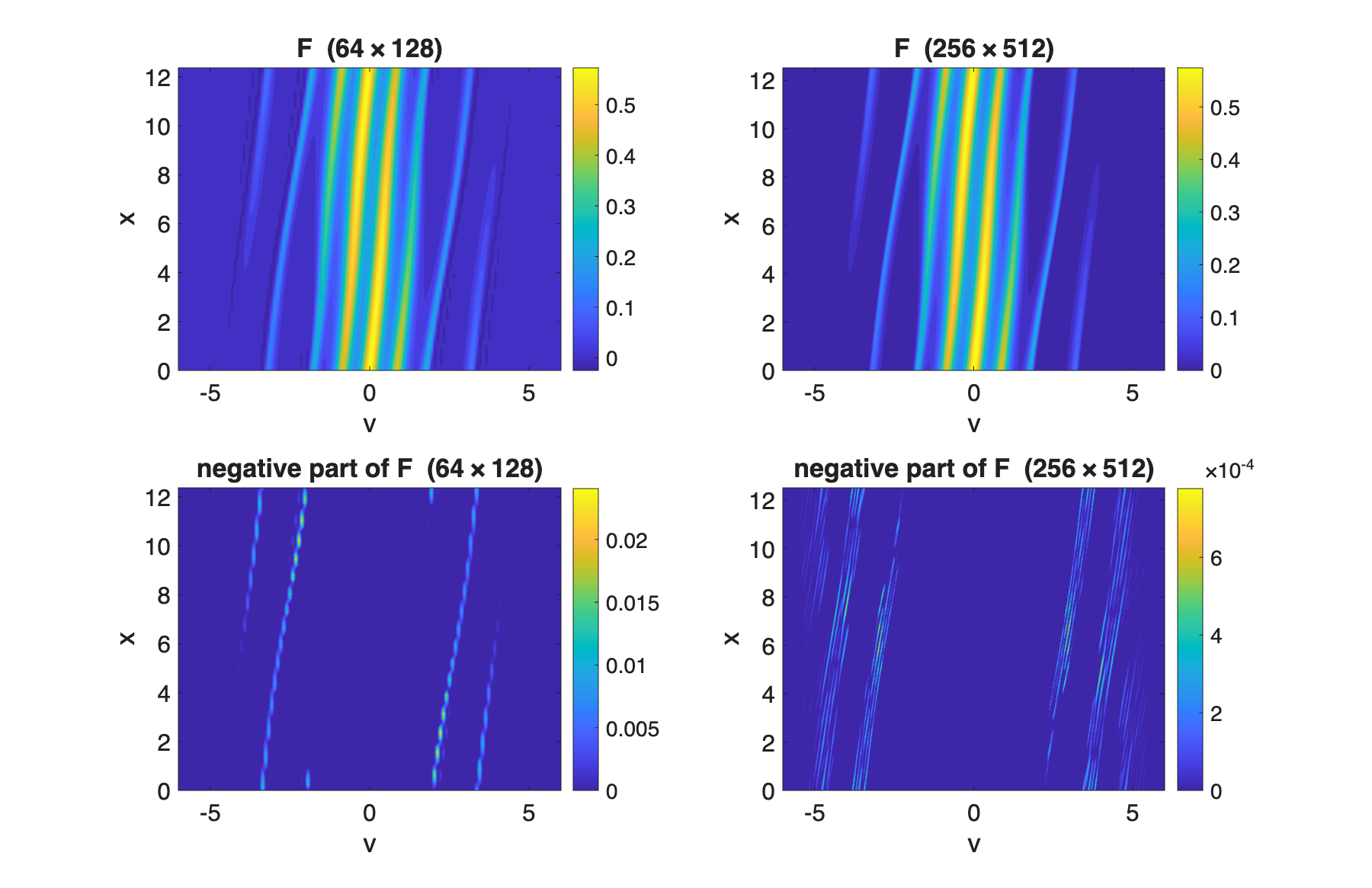}
        \caption{Low-rank Landau damping solution $F$ (top) and its negative part $-\min\{F,\bm 0\}$ (bottom) on the $64\times 128$ (left) and $256\times 512$ (right) grids, in the style of the left column of Figure~\ref{1D1V data}. The negative entries, introduced by the SVD-type truncation, concentrate along the filament edges and become finer and smaller in magnitude under refinement.}
        \label{fig:data_sizes}
    \end{center}
\end{figure}

\begin{table}[htbp]
\begin{center}

\begin{tabular}{lcccccc}
\toprule
 & & \multicolumn{2}{c}{convex (DR splitting)} & \multicolumn{2}{c}{non-convex (TAP)} & \\
\cmidrule(lr){3-4}\cmidrule(lr){5-6}
grid $N_x\times N_v$ & $mn$ & full SVD (s) & DR / iter (s) & tangent proj.\ (s) & TAP / iter (s) & ratio \\
\midrule
$32\times 64$   & $2{,}048$   & $1.9\times10^{-4}$ & $2.6\times10^{-4}$ & $3.1\times10^{-4}$ & $3.0\times10^{-4}$ & $0.9$ \\
$64\times 128$  & $8{,}192$   & $5.5\times10^{-4}$ & $8.9\times10^{-4}$ & $4.4\times10^{-4}$ & $4.9\times10^{-4}$ & $1.8$ \\
$128\times 256$ & $32{,}768$  & $3.1\times10^{-3}$ & $4.1\times10^{-3}$ & $6.3\times10^{-4}$ & $8.7\times10^{-4}$ & $4.7$ \\
$256\times 512$ & $131{,}072$ & $1.7\times10^{-2}$ & $2.8\times10^{-2}$ & $1.1\times10^{-3}$ & $1.6\times10^{-3}$ & $18.0$ \\
\bottomrule
\end{tabular}

\caption{Per-iteration cost scaling across grid sizes for the Landau damping problem (median of three trials, convergence diagnostics excluded). For the convex problem we report the cost of one full $m\times n$ SVD (the dominant kernel) and the full per-iteration cost of the Douglas--Rachford splitting. For the non-convex problem we report the cost of one tangent-space projection (the $2r\times 2r$-SVD kernel) and the full per-iteration cost of TAP, with $r = 20$ fixed. The last column is the per-iteration ratio, DR\,/\,TAP. The negativity measure $\norm{\min\{\bm A,\bm 0\}}_F/\norm{\bm A}_F$ ranges from $2.8\%$ to $11.3\%$ and $\rank(\bm A_2)$ from $23$ to $41$ across the four grids.}
\label{tab:scaling}
\end{center}
\end{table}

For convenience we write $m = N_x$ and $n = N_v$, so the low-rank field is an $m\times n$ matrix. The per-iteration cost of every convex algorithm is dominated by a single full $m\times n$ SVD, which costs $\mathcal{O}(mn\min(m,n))$, whereas a TAP iteration is dominated by one tangent-space projection (a $2r\times 2r$ SVD plus $\mathcal{O}(mnr)$ work, Algorithm~\ref{TP for non-convex problem}) with a fixed rank $r$. Since $n = 2m$ here, we have $\min(m,n) = m$ and $mn = 2m^2$, so the convex per-iteration cost grows as $\mathcal{O}\!\left((mn)^{3/2}\right)$ while the TAP per-iteration cost grows only as $\mathcal{O}(mn)$. For each grid size we time both the dominant kernel and the full algorithm (Douglas--Rachford splitting for the convex problem, with $a = 10^3$, and TAP for the non-convex problem, with $r = 20$). To make the comparison fair we strip the convergence diagnostics from both inner loops, so the reported times reflect only the algorithmic work. The results are reported in Table~\ref{tab:scaling} and Figure~\ref{fig:scaling}.

\begin{figure}[htbp]
    \begin{center}
        \includegraphics[width=\textwidth]{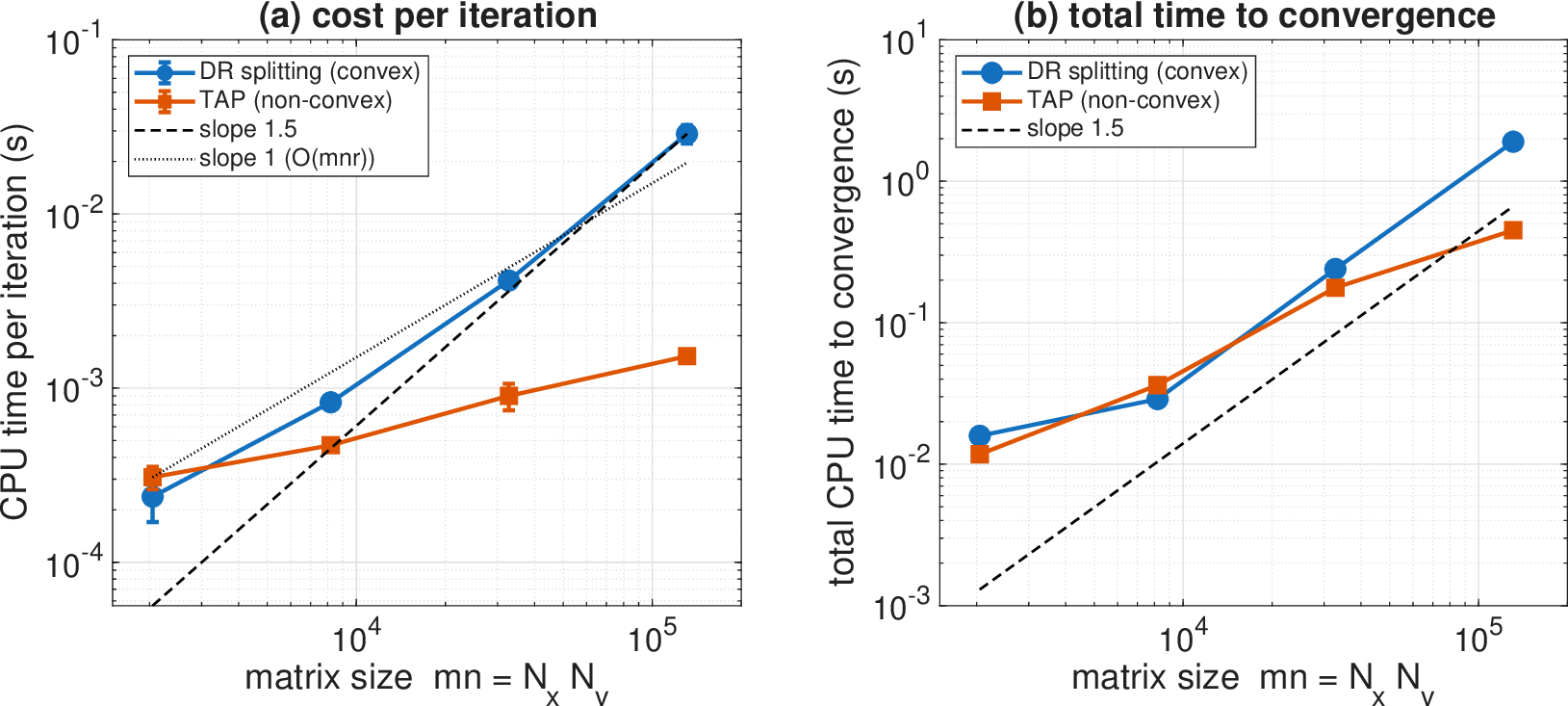}
        \caption{Cost scaling versus matrix size $mn = N_x N_v$ (log--log). (a) Per-iteration cost of the Douglas--Rachford splitting (convex) and the TAP algorithm (non-convex), plotted as the median over $20$ runs with error bars of $\pm 1$ standard deviation. The DR cost follows the predicted $(mn)^{3/2}$ slope at the larger sizes, while the TAP cost grows \emph{sub-linearly} (empirical slope $\approx 0.4$), below its asymptotic $\mathcal{O}(mnr)$ rate (slope $1$). In this size range a TAP iteration is dominated by the size-independent cost of the tangent-space projection (the $2r\times 2r$ SVD and the QR of the tangent factors, $\mathcal{O}(r^3+(m+n)r^2)$) rather than by the $\mathcal{O}(mnr)$ matrix products, which take only microseconds, so the linear regime would set in only at much larger $mn$. The per-iteration gap between DR and TAP nonetheless reaches about $18\times$ at $256\times 512$. (b) Total wall-clock time to convergence (relative change of the iterate below $10^{-8}$, with the DR step size tuned per grid): the two methods are comparable on the smaller grids, while TAP is about $4\times$ faster at $256\times 512$. With its tuned step size DR needs fewer iterations than TAP, but its per-iteration full SVD makes each iteration much more expensive.}
        \label{fig:scaling}
    \end{center}
\end{figure}

The measured costs confirm the predicted per-iteration scaling (Figure~\ref{fig:scaling}(a)). The convex per-iteration cost follows the $(mn)^{3/2}$ reference slope at the larger grids (the DR per-iteration time increases by a factor of about $6.9$ over the last fourfold increase in $mn$, approaching the asymptotic factor of $8$), while the TAP per-iteration cost grows much more slowly. The per-iteration advantage of TAP therefore grows monotonically with the problem size, from being slightly more expensive on the smallest $32\times 64$ grid (where the highly optimized full SVD of a tiny matrix is cheaper than the TAP projection overhead) to being $18$ times cheaper on the $256\times 512$ grid.

The total cost is the per-iteration cost times the number of iterations to convergence, so we also examine how the iteration count scales with the problem size. We run both algorithms until the relative change of the iterate $\norm{\bm X_k - \bm X_{k-1}}_F/\norm{\bm X_{k-1}}_F$ falls below $10^{-8}$, with $a = 10^3$ and $r = 20$ for TAP. For DR we tune the step size at each grid, so that DR is compared at its best at every size. The optimal step grows with the problem (from $\eta \approx 2a$ on the coarsest grid to $\eta \approx 11a$ on the finest, with relaxation $\alpha$ between $0.90$ and $0.94$). The two algorithms reach comparable relative corrections $\norm{\bm X}_F/\norm{\bm A}_F$ (last two columns of Table~\ref{tab:conv}), so the comparison is fair. With its step size tuned, DR converges in markedly fewer iterations than TAP, and the gap widens with refinement (DR from $30$ to $82$ iterations versus TAP from $32$ to $264$, Table~\ref{tab:conv}). Each DR iteration nevertheless requires a full SVD, whereas a TAP iteration requires only a $2r\times 2r$ SVD, so DR's per-iteration cost grows much faster (Figure~\ref{fig:scaling}(a)). The two effects offset in the total wall-clock time (Figure~\ref{fig:scaling}(b)), which is comparable for the two methods on the three smaller grids (DR is even slightly faster at $64\times 128$) while TAP is about $4\times$ faster on the finest $256\times 512$ grid. Since the per-iteration cost gap scales as $\min(m,n)/r = m/r$ while DR's iteration advantage grows only mildly, we expect TAP's total-time advantage to keep widening on still finer grids.

\begin{table}[htbp]
\begin{center}

\begin{tabular}{lccccccc}
\toprule
 & & \multicolumn{2}{c}{iterations to conv.} & \multicolumn{2}{c}{total time (s)} & \multicolumn{2}{c}{$\norm{\bm X}_F/\norm{\bm A}_F$} \\
\cmidrule(lr){3-4}\cmidrule(lr){5-6}\cmidrule(lr){7-8}
grid $N_x\times N_v$ & $mn$ & DR & TAP & DR & TAP & DR & TAP \\
\midrule
$32\times 64$   & $2{,}048$   & $30$  & $32$  & $0.016$ & $0.012$ & $11.6\%$ & $11.6\%$ \\
$64\times 128$  & $8{,}192$   & $31$  & $74$  & $0.029$ & $0.036$ & $5.8\%$  & $6.2\%$ \\
$128\times 256$ & $32{,}768$  & $55$  & $196$ & $0.24$  & $0.18$  & $8.1\%$  & $9.5\%$ \\
$256\times 512$ & $131{,}072$ & $82$  & $264$ & $1.90$  & $0.45$  & $2.8\%$  & $3.7\%$ \\
\bottomrule
\end{tabular}

\caption{Iterations to convergence (relative change of the iterate below $10^{-8}$) and total wall-clock time for DR splitting and TAP ($r = 20$) across grid sizes, with $a = 10^3$. The DR step size is tuned per grid (the optimum grows with the grid: $\eta/a = 2.0, 2.5, 7.1, 11.2$, relaxation $\alpha = 0.90$ on the two coarse grids and $0.94$ on the two fine grids). With its step size tuned, DR converges in fewer iterations than TAP, increasingly so as the grid is refined. Each DR iteration nevertheless requires a full SVD, so its per-iteration cost grows faster, and the total time favors TAP only on the finest grid. The last two columns show that the two algorithms reach comparable relative corrections. CPU times are indicative and machine-dependent.}
\label{tab:conv}
\end{center}
\end{table}

\begin{remark}[Accelerating the per-iteration SVD]
The per-iteration cost of the convex algorithms is dominated by a single full \(m\times n\) SVD (Table~\ref{tab:scaling}), so it is natural to ask whether a cheaper low-rank factorization can replace it. A maxvol-based cross (CUR) approximation \cite{goreinov2010}, which samples a few rows and columns without ever forming the full matrix, is slower than the built-in SVD at every size we tested (for instance about \(8\) times slower at \(128\times 256\)). The reason is that it is designed to minimize the number of entry evaluations, an advantage only when the entries are expensive to compute, as for implicit functions or high-dimensional tensors, whereas our matrices are explicit and stored in memory, so a single level-3 BLAS SVD is faster.

A randomized SVD \cite{halko2011} is more promising, but its benefit is strongly rank-dependent, and this is the essential point. Its cost scales as \(\mathcal{O}(mnr)\), so it beats the full SVD, which costs \(\mathcal{O}(mn\min(m,n))\), only when the numerical rank \(r\) is a small fraction of \(\min(m,n)\). This holds for the data matrix \(\bm A\), which is genuinely low rank: there the randomized SVD reproduces the same singular triplets and is faster on the finer grids, by about three times on the finest (Table~\ref{tab:svdaccel}, upper block). It does \emph{not} hold for the matrices that the DR iteration actually factors. Each iteration applies the SVD not to \(\bm A\) but to a proximal input \(\bm W\), a dense intermediate matrix formed from the current iterates that is numerically full rank, and the soft-thresholded output retains a large fraction of the spectrum, roughly half of \(\min(m,n)\) (Table~\ref{tab:svdaccel}, lower block). On these matrices the same randomized SVD gives no speedup and is in fact slower at every size. Replacing the SVD by a randomized one therefore does not accelerate DR, even though it reproduces the same iterates to controllable accuracy. The low \(\mathcal{O}(mnr)\) per-iteration cost is instead achieved by the tangent-space method, which keeps the iterate in rank-\(r\) factored form and never forms or factors a full-rank matrix.
\end{remark}

\begin{table}[htbp]
\begin{center}

\begin{tabular}{lrccc}
\toprule
grid $N_x\times N_v$ & $r$ & full SVD (s) & randomized SVD (s) & speedup \\
\midrule
\multicolumn{5}{l}{\emph{data matrix} \(\bm A\) (genuinely low rank)}\\
$64\times 128$  & $31$  & $2.9\times10^{-4}$ & $6.6\times10^{-4}$ & $0.44\times$ \\
$128\times 256$ & $43$  & $2.2\times10^{-3}$ & $1.6\times10^{-3}$ & $1.35\times$ \\
$256\times 512$ & $44$  & $9.5\times10^{-3}$ & $3.2\times10^{-3}$ & $3.02\times$ \\
\midrule
\multicolumn{5}{l}{\emph{DR proximal matrix} \(\bm W\) (numerically full rank)}\\
$64\times 128$  & $58$  & $3.6\times10^{-4}$ & $1.3\times10^{-3}$ & $0.29\times$ \\
$128\times 256$ & $84$  & $2.2\times10^{-3}$ & $3.9\times10^{-3}$ & $0.58\times$ \\
$256\times 512$ & $121$ & $1.2\times10^{-2}$ & $1.3\times10^{-2}$ & $0.95\times$ \\
\bottomrule
\end{tabular}

\caption{Rank-dependence of the randomized-SVD speedup. The same randomized SVD (subspace iteration with oversampling) is applied to two matrices at each grid: the data matrix \(\bm A\), which is genuinely low rank, and a representative Douglas--Rachford proximal input \(\bm W\), a dense intermediate matrix formed from the current iterates that is numerically full rank. Here \(r\) is the numerical rank of \(\bm A\) and the retained rank of the thresholded \(\bm W\), and the randomized SVD targets this \(r\). The speedup is the full-SVD time divided by the randomized-SVD time (fastest of seven trials), so a value above one means the randomized SVD is faster. On the low-rank \(\bm A\) it is faster on the finer grids, whereas on the full-rank \(\bm W\) it is slower at every size, because its \(\mathcal{O}(mnr)\) cost helps only when \(r \ll \min(m,n)\). CPU times are indicative and machine-dependent.}
\label{tab:svdaccel}
\end{center}
\end{table}

\subsection{\texorpdfstring{Application as a positivity limiter in the LoMaC scheme}{Application as a positivity limiter in the LoMaC scheme}}\label{sec:limiter}

So far we have applied the correction to a single fixed matrix. We now use it as a \emph{positivity limiter} inside the time-dependent LoMaC low-rank scheme of \cite{guo2022local}, which conserves total mass, momentum \emph{and} energy at the discrete level. As the low-rank solution is advanced, the SVD-type truncation at each step introduces small negative entries. We apply the correction periodically (every $k$ steps and on the final step, each application an SVD-scale computation) to remove the unphysical negative entries while preserving the macroscopic moments, since its constraint $\bm X\bm B = \bm 0$ leaves the three conserved densities (mass, momentum and energy) untouched. We compare two correction algorithms as the limiter. The first is the convex Douglas--Rachford splitting (DR, Algorithm~\ref{DR for convex problem}), which imposes \emph{no a-priori rank} on the correction: the corrected field takes whatever rank nonnegativity requires, and its rank is reduced only by the scheme's own conservative truncation $\mathcal{T}_{\varepsilon_{\rm tr}}$ at a tolerance $\varepsilon_{\rm tr}$. (For DR we use the step size $\eta = 10\,a$, found near-optimal by a step-size sweep on a representative correction subproblem: $83$ iterations to drive $\norm{\bm X\bm B}_F$ below $10^{-10}$, versus $190$ at $\eta=10^{0.5}a$ and $602$ at $\eta=a$.) The second is the rank-constrained TAP (Algorithm~\ref{TAP for non-convex problem}). The essential point is that its rank must \emph{not} be fixed at some foreign value. Instead, we let it track the solution, taking $r = N+3$ where $N$ is the current numerical rank of the low-rank solution at each application. This gives TAP the rank the field already carries, plus a small margin for the extra rank that nonnegativity introduces, so that, like DR, it is never artificially capped. (A fixed rank, e.g.\ $r=20$, would cap the attainable nonnegativity, as we note below.)

We test the limiter on the bump-on-tail instability used to validate the LoMaC scheme (Example~5.4 of \cite{guo2022local}), a well-resolved, genuinely low-rank run of the same family as the data in Figure~\ref{1D1V data}, with initial datum
\begin{equation*}
  f_0(x,v) = \left(1+\alpha\cos(kx)\right)\!\left(n_p\,e^{-v^2/2}+n_b\,e^{-(v-u)^2/(2 v_t^2)}\right),
\end{equation*}
$\alpha=0.04$, $k=0.3$, $n_p=\tfrac{9}{10\sqrt{2\pi}}$, $n_b=\tfrac{2}{10\sqrt{2\pi}}$, $u=4.5$, $v_t=0.5$, on $[0,2\pi/k]\times[-13,13]$, integrated to $T=30$ with $\varepsilon=10^{-4}$, using the validated LoMaC discretization (linear fifth-order upwind in $x$, $\mathrm{CFL}=0.1$ on both grids, so that the two runs share the same time-step criterion). We run two grids, $64\times 128$ and $128\times 256$, and apply the limiter every $k=1000$ steps. Table~\ref{tab:limiter} and Figure~\ref{fig:limiter} summarize the results.

In this regime the LoMaC solution is genuinely low rank (numerical rank $28$ on $64\times 128$ and $49$ on $128\times 256$ at $T=30$), and the scheme conserves total mass, momentum \emph{and} energy to machine precision ($\sim 10^{-14}$ for all three). The SVD truncation nonetheless leaves a negativity $\norm{\min\{\bm F,\bm 0\}}_F/\norm{\bm F}_F \approx 2$--$3\times 10^{-4}$. The limiter removes it, and because the correction satisfies $\bm X\bm B=\bm 0$ it does so without disturbing the three conserved quantities. Here the two correction algorithms differ in an instructive way (Table~\ref{tab:limiter}): TAP keeps $\bm X_k\bm B = \bm 0$ \emph{exactly at every iterate}, so the limited run conserves mass, momentum and energy to machine precision ($\sim 10^{-14}$), just like the unlimited LoMaC scheme. DR enforces $\bm X\bm B = \bm 0$ only in the limit, so each application perturbs the moments by its orthogonality residual $\norm{\bm X\bm B}_F$, and conservation is preserved only to that level ($\sim 10^{-10}$ at $\varepsilon_{\rm tr}=10^{-4}$, improving to $\sim 10^{-12}$ as the correction is converged more tightly at $\varepsilon_{\rm tr}=10^{-10}$). For exact conservation, then, the rank-constrained TAP limiter is preferable, which mirrors the constraint-preservation property noted in Section~\ref{discussion}.

The other point concerns the interplay between nonnegativity and rank, and it explains why the stored solution is not nonnegative to machine zero by default. The correction itself enforces nonnegativity \emph{absolutely}: the corrected field satisfies $\bm A+\bm X \geqslant \bm 0$ to machine precision \emph{before} truncation (for DR, $\bm X$ is the exact pointwise clip $\bm X = \max\{-\bm A,\,\cdot\,\}$, so $\bm A+\bm X\geqslant\bm 0$ identically). Negativity reappears \emph{only} when this corrected field is re-expressed in the rank-bounded low-rank format. A nonnegative field generically has higher rank than the smooth, slightly negative one, so the conservative truncation $\mathcal{T}_{\varepsilon_{\rm tr}}$ must discard a singular-value tail of norm $O(\varepsilon_{\rm tr})$. Because the corrected field sits exactly at zero where it had been negative, discarding this tail perturbs it by $O(\varepsilon_{\rm tr})$ and pushes those entries just below zero, leaving a negativity of magnitude $O(\varepsilon_{\rm tr})$. The negativity floor of the stored low-rank solution is therefore set by the truncation tolerance $\varepsilon_{\rm tr}$, not by the correction: at $\varepsilon_{\rm tr}=10^{-4}$ (the scheme's own truncation level) the negativity is $\sim 2\times 10^{-6}$ at a modestly higher rank ($49\to 51$--$53$ on $128\times 256$), while tightening to $\varepsilon_{\rm tr}=10^{-10}$ drives it toward machine precision ($\sim 10^{-12}$, $\min \bm F \sim -10^{-11}$ to $-10^{-12}$) at a higher rank ($49\to 84$ for DR, $49\to 91$ for TAP). Provided the rank is allowed to grow, both algorithms reach comparable floors (Table~\ref{tab:limiter} and Figure~\ref{fig:limiter}). This is also why letting TAP's rank track the solution ($r=N+3$) matters: a fixed externally-imposed rank caps the attainable nonnegativity (with $r=20$ the rank-$20$ correction stalls at $\min(\bm A+\bm X)\approx-1.4\times 10^{-10}$ regardless of $\varepsilon_{\rm tr}$), whereas the adaptive rank, like the unconstrained convex correction, removes that cap. In short, the residual negativity is governed by how much rank one keeps in the low-rank format, not by the correction itself.

Figure~\ref{fig:limiter}(a) shows the negativity of the stored low-rank field (that is, after the correction \emph{and} the re-truncation) just after each limiter application. The DR and TAP curves overlap at each $\varepsilon_{\rm tr}$, reaching $\sim 2\times 10^{-6}$ at $\varepsilon_{\rm tr}=10^{-4}$ and $\sim 10^{-12}$ at $\varepsilon_{\rm tr}=10^{-10}$, confirming that the floor follows $\varepsilon_{\rm tr}$ rather than the correction. (Between applications the per-step truncation at $\varepsilon$ lets the negativity regrow, bounded by the no-limiter level, so a moderate $k$ controls it while the final output, limited on the last step, is nonnegative to $\varepsilon_{\rm tr}$.) Figure~\ref{fig:limiter}(b) shows the rank: both limiters hold the solution at the higher rank that nonnegativity demands, increasingly so as $\varepsilon_{\rm tr}$ is tightened. The cost is modest and comparable for the two methods: with $k=1000$ the limiter performs $15$ ($64\times 128$) or $30$ ($128\times 256$) correction calls, raising the wall-clock time on $64\times 128$ from $23$ to $28$\,s and on $128\times 256$ from $174$\,s to $208$--$230$\,s (a $20$--$30\%$ overhead, dominated by the higher rank sustained between calls rather than by the correction solves). TAP is the cheaper of the two here, and its per-iteration advantage over DR grows on finer grids (Section~\ref{sec:scaling}).

\begin{table}[htbp]
\begin{center}

\resizebox{\textwidth}{!}{%
\begin{tabular}{llcccccc}
\toprule
mesh & limiter & mass dev. & mom.\ dev. & energy dev. & negativity & $\rank$ & CPU (s) \\
\midrule
$64\times 128$  & none                                  & $5.6\times10^{-15}$ & $2.5\times10^{-14}$ & $1.5\times10^{-14}$ & $1.8\times10^{-4}$  & $28$  & $23$  \\
$64\times 128$  & DR,\ $\varepsilon_{\rm tr}=10^{-4}$       & $4.4\times10^{-14}$ & $8.3\times10^{-14}$ & $1.6\times10^{-13}$ & $1.1\times10^{-6}$  & $27$  & $28$  \\
$64\times 128$  & TAP,\ $\varepsilon_{\rm tr}=10^{-4}$      & $6.6\times10^{-15}$ & $2.5\times10^{-14}$ & $1.6\times10^{-14}$ & $1.5\times10^{-6}$  & $28$  & $28$  \\
$64\times 128$  & DR,\ $\varepsilon_{\rm tr}=10^{-10}$      & $2.7\times10^{-14}$ & $2.2\times10^{-13}$ & $5.5\times10^{-13}$ & $2.1\times10^{-12}$ & $54$  & $28$  \\
$64\times 128$  & TAP,\ $\varepsilon_{\rm tr}=10^{-10}$     & $3.8\times10^{-15}$ & $2.6\times10^{-14}$ & $1.7\times10^{-14}$ & $2.8\times10^{-12}$ & $53$  & $28$  \\
\midrule
$128\times 256$ & none                                  & $9.8\times10^{-15}$ & $4.0\times10^{-14}$ & $3.5\times10^{-14}$ & $3.0\times10^{-4}$  & $49$  & $174$ \\
$128\times 256$ & DR,\ $\varepsilon_{\rm tr}=10^{-4}$       & $1.7\times10^{-11}$ & $1.1\times10^{-10}$ & $6.3\times10^{-11}$ & $2.0\times10^{-6}$  & $53$  & $222$ \\
$128\times 256$ & TAP,\ $\varepsilon_{\rm tr}=10^{-4}$      & $9.9\times10^{-15}$ & $3.3\times10^{-14}$ & $3.0\times10^{-14}$ & $2.0\times10^{-6}$  & $51$  & $208$ \\
$128\times 256$ & DR,\ $\varepsilon_{\rm tr}=10^{-10}$      & $5.6\times10^{-13}$ & $2.2\times10^{-12}$ & $6.7\times10^{-13}$ & $2.9\times10^{-13}$ & $84$  & $230$ \\
$128\times 256$ & TAP,\ $\varepsilon_{\rm tr}=10^{-10}$     & $9.3\times10^{-15}$ & $3.4\times10^{-14}$ & $2.9\times10^{-14}$ & $1.4\times10^{-12}$ & $91$  & $208$ \\
\bottomrule
\end{tabular}}

\caption{Correction as a positivity limiter in the LoMaC low-rank scheme \cite{guo2022local} for the bump-on-tail run, $T=30$, applied every $1000$ steps, comparing the convex DR limiter (no rank cap) with the rank-constrained TAP limiter at rank $r=N+3$ ($N$ = current solution rank). The negativity is $\norm{\min\{\bm F,\bm 0\}}_F/\norm{\bm F}_F$ for the limited (post-application) solution. The LoMaC scheme conserves mass, momentum and energy to machine precision. TAP keeps $\bm X_k\bm B=\bm 0$ exactly at every iterate, so the limited run conserves all three to machine precision. DR enforces $\bm X\bm B=\bm 0$ only at convergence, so it preserves conservation to its orthogonality residual ($\sim 10^{-10}$ at $\varepsilon_{\rm tr}=10^{-4}$, $\sim 10^{-12}$ at $\varepsilon_{\rm tr}=10^{-10}$). Both drive the negativity to the re-truncation floor $\varepsilon_{\rm tr}$, with the kept rank rising as $\varepsilon_{\rm tr}$ is tightened. Both grids use $\mathrm{CFL}=0.1$. The per-step cost of the LoMaC integration scales as $\mathcal{O}((m+n)r^2)$, linear in the grid size and quadratic in the rank $r$, so it remains a low-rank cost even as the limiter raises $r$. Enforcing nonnegativity increases the rank only modestly and transiently: at $\varepsilon_{\rm tr}=10^{-4}$, the scheme's own truncation level, the rank is essentially unchanged, whereas at $\varepsilon_{\rm tr}=10^{-10}$ it roughly doubles (for instance $28\to 53$ on $64\times 128$ and $49\to 91$ on $128\times 256$), but this higher rank is sustained only briefly after each application before the per-step truncation relaxes it. The time-averaged overhead is therefore only $20$--$30\%$ and stays far below the full-grid $\mathcal{O}(mn\min(m,n))$ cost. CPU times are indicative and machine-dependent.}
\label{tab:limiter}
\end{center}
\end{table}

\begin{figure}[htbp]
    \begin{center}
        \includegraphics[width=\textwidth]{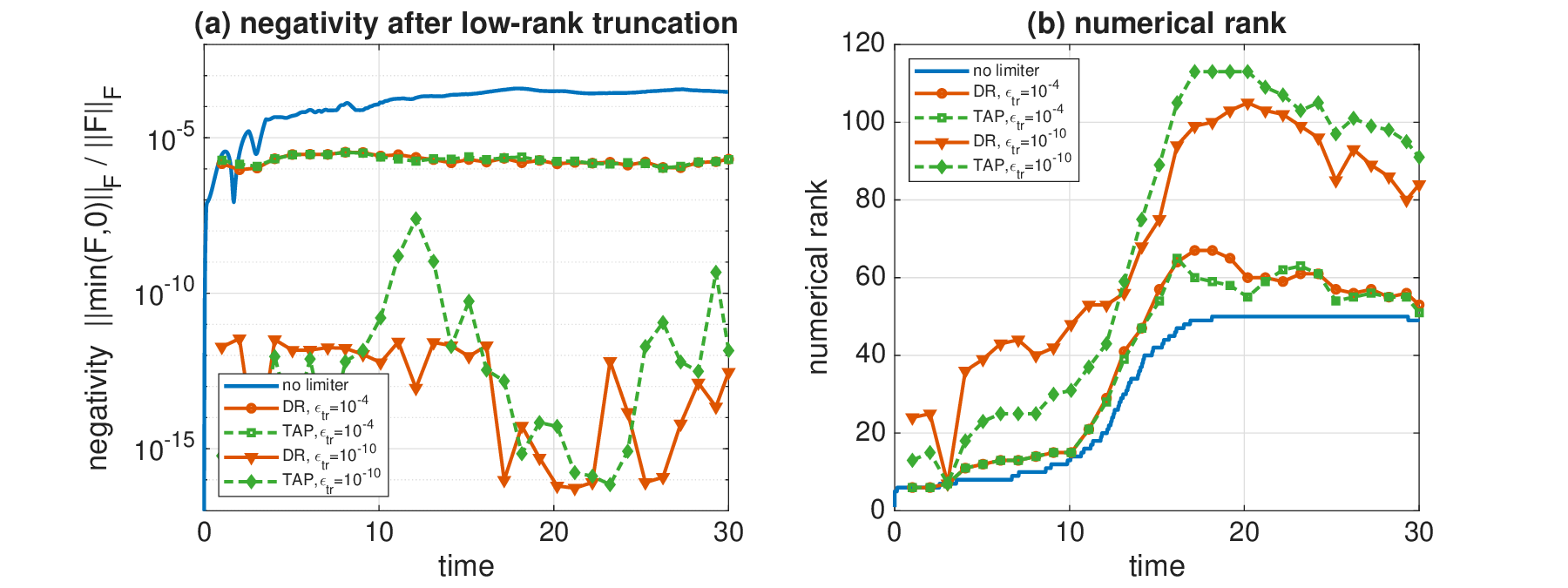}
        \caption{Positivity limiter in the LoMaC low-rank scheme for the bump-on-tail run ($128\times 256$, $T=30$), applied every $1000$ steps, comparing the convex DR limiter (no rank cap) with the rank-constrained TAP limiter at rank $r=N+3$. (a) Negativity $\norm{\min\{\bm F,\bm 0\}}_F/\norm{\bm F}_F$ of the stored low-rank field versus time. The correction enforces nonnegativity to machine precision \emph{before} truncation ($\bm A+\bm X\geqslant\bm 0$ identically). The residual negativity plotted here is reintroduced \emph{solely} by re-expressing the corrected field in the rank-bounded low-rank format, by discarding singular values below $\varepsilon_{\rm tr}$, so it scales with $\varepsilon_{\rm tr}$. Without the limiter it grows to $\sim 3\times 10^{-4}$. With either limiter each application resets it to this re-truncation floor, $\sim 2\times 10^{-6}$ at $\varepsilon_{\rm tr}=10^{-4}$ and $\sim 10^{-12}$ at $\varepsilon_{\rm tr}=10^{-10}$ (DR and TAP curves overlap). (b) Numerical rank: enforcing nonnegativity raises the kept rank above the no-limiter value ($49$), more so as $\varepsilon_{\rm tr}$ is tightened ($\approx 51$--$53$ at $\varepsilon_{\rm tr}=10^{-4}$, $\approx 84$--$91$ at $\varepsilon_{\rm tr}=10^{-10}$). The rank is reduced only by the scheme's conservative truncation, never by a mechanical cap, and DR and adaptive-rank TAP track each other.}
        \label{fig:limiter}
    \end{center}
\end{figure}

One might worry that the numerical rank of the limited solution, which reaches a sizeable fraction of $\min(N_x,N_v)$ on the $128\times 256$ grid of Figure~\ref{fig:limiter}, would approach full rank as the mesh is refined. It does not: the near-full appearance is an artifact of the small $\min(N_x,N_v)$ at this resolution, not of the low-rank structure. To exhibit the trend cheaply we run the same test on the coarser $32\times 64$ and $64\times 128$ grids (Figure~\ref{fig:ranktrend}). The peak numerical rank grows only sub-linearly with the grid (for the no-limiter run it is $18$, $28$, $50$ across $32\times 64$, $64\times 128$, $128\times 256$, each refinement multiplying it by well under the factor $2$ that a full-rank field would require), so as a fraction of the maximum $\min(N_x,N_v)$ it \emph{decreases} under refinement. On the coarsest grid the tight-tolerance ($\varepsilon_{\rm tr}=10^{-10}$) limiter is pinned at the full-rank ceiling, whereas on $128\times 256$ it already sits at $80$--$88\%$ of it. Doubling the resolution therefore makes the rank a smaller fraction of full, not larger, confirming that the limited solution remains genuinely low rank.

\begin{figure}[htbp]
    \begin{center}
        \includegraphics[width=\textwidth]{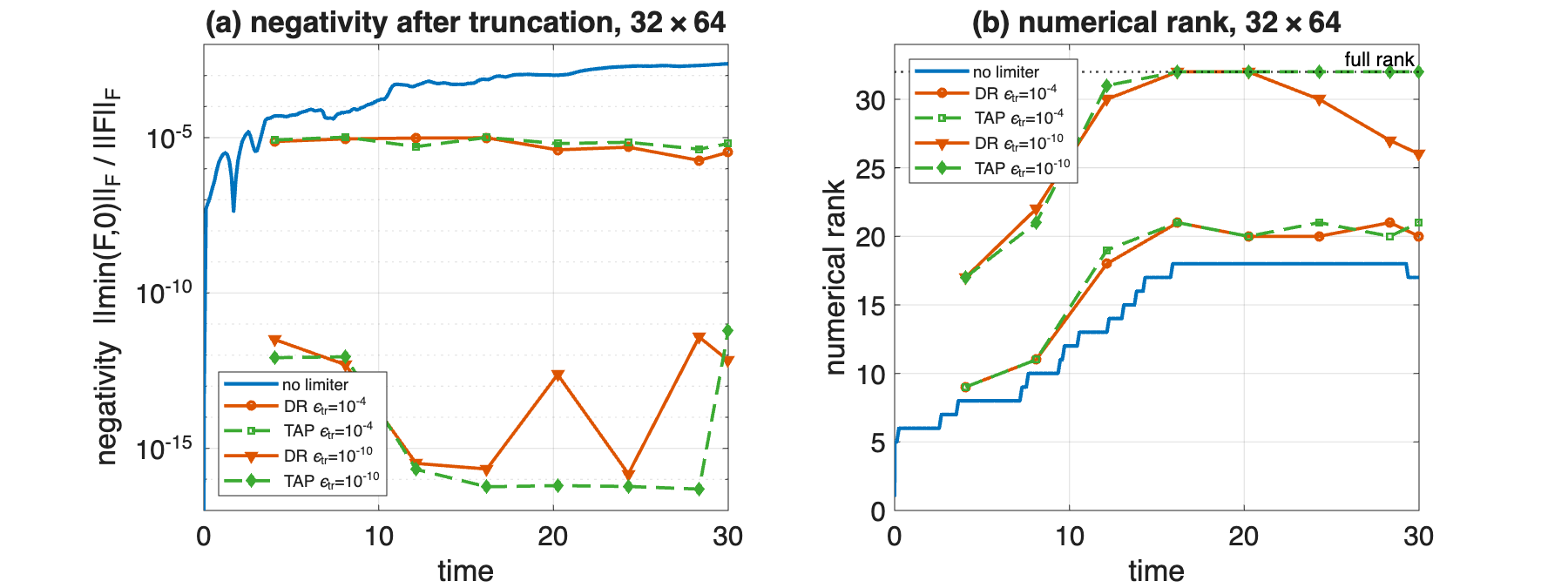}\\[4pt]
        \includegraphics[width=0.62\textwidth]{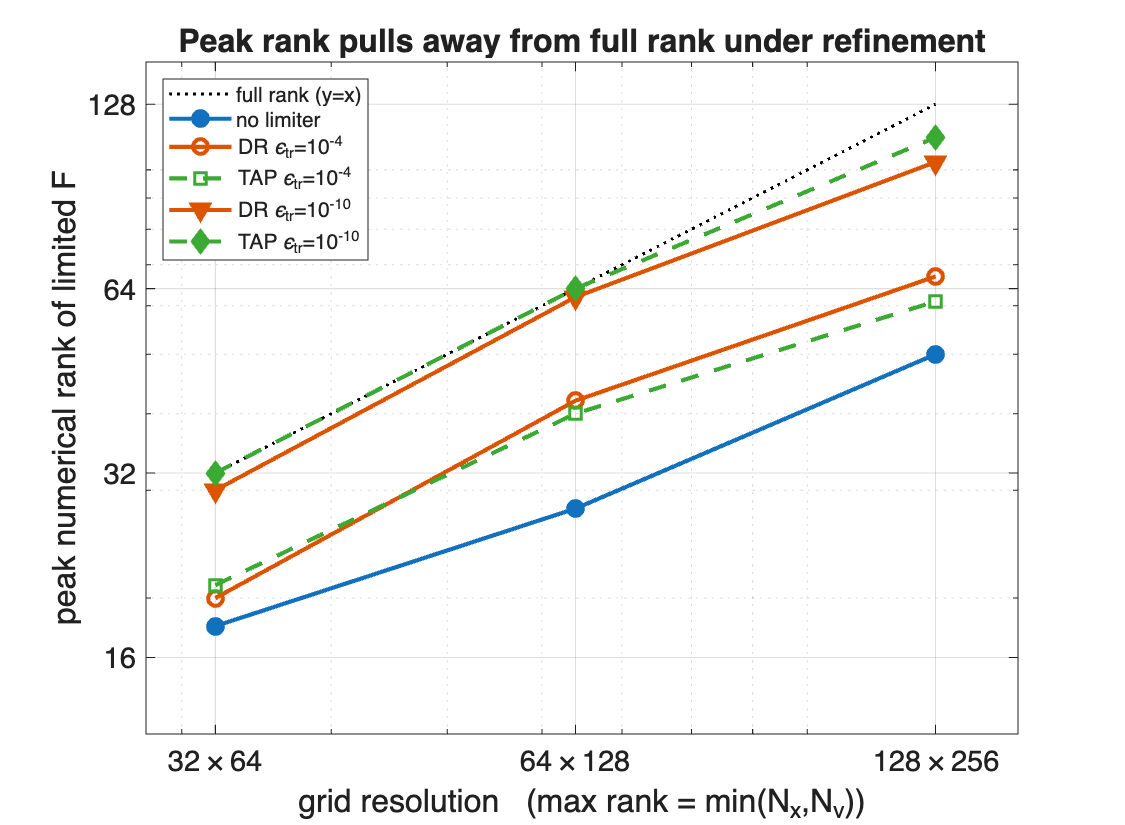}
        \caption{Rank of the limited solution versus mesh resolution, addressing whether it approaches full rank under refinement. Top: the coarse-grid ($32\times 64$) counterpart of Figure~\ref{fig:limiter}, showing the negativity (left) and the numerical rank (right) versus time; on this small grid the tight-tolerance limiter is pinned at the full-rank ceiling $\min(N_x,N_v)=32$ (dotted). Bottom: peak numerical rank of the limited solution versus grid resolution for each configuration, against the full-rank line $y=\min(N_x,N_v)$ (dotted). Every curve pulls away from the full-rank line as the mesh is refined, since the absolute rank grows sub-linearly, so the near-full rank seen in Figure~\ref{fig:limiter} is a coarse-grid artifact and doubling the resolution reduces the rank as a fraction of the maximum. All runs use $\mathrm{CFL}=0.1$.}
        \label{fig:ranktrend}
    \end{center}
\end{figure}

\subsection{Discussion}\label{discussion}

We summarize the main observations from the experiments of this section, which comprise a fixed-size correction problem (the $64\times 128$ Landau data, Table~\ref{tab:correction} and Figures~\ref{fig:convex-1}--\ref{1D1V TAP}), a cost-scaling study across four grid sizes (Section~\ref{sec:scaling}), and the use of the correction as a positivity limiter in a time-dependent solver (Section~\ref{sec:limiter}).

\paragraph{Approximation quality.} All the algorithms produce corrections of comparable quality. For $a\geqslant 10$ in the convex formulation and for every rank constraint $r$ tested in TAP, the relative correction $\norm{\bm X}_F/\norm{\bm A}_F$ is close to the theoretical lower bound (here $3.08\%$) and decreases monotonically as $a$ or $r$ increases. All convex algorithms reach the same global minimizer ($3.29\%$, rank $33$), reflecting the global optimality of the convex formulation, and TAP attains a comparable correction ($3.31\%$) at the prescribed lower rank $r=20$. The penalty $a$ and the rank bound $r$ play the same role of trading correction quality against the rank of $\bm X$. The naive low-rank baseline matches the rank and orthogonality constraints but cannot enforce nonnegativity ($\min(\bm A+\bm X)<0$), and its smaller relative correction ($2.93\%$, below the $3.08\%$ bound) is possible only because it violates nonnegativity, which is exactly the property the optimization-based corrections are designed to deliver.

\paragraph{Convergence speed.} On the fixed-size problem, Douglas--Rachford splitting converges in the fewest iterations among the convex algorithms (about $65$ at $a=1000$, Table~\ref{tab:correction}), once its step size and relaxation are tuned (Figure~\ref{fig:convex-2}). The dual methods are slower. Among them the gradient-based restart performs best, the L-BFGS method needs few iterations but an expensive line search, and the PR+ conjugate gradient is the slowest and does not meet the tolerance within $1000$ iterations. The tangent-space accelerated alternating projection converges in the fewest iterations overall (about $45$).

\paragraph{Per-iteration cost and scaling.} The two formulations differ fundamentally in per-iteration cost: every convex iteration requires a full $m\times n$ SVD ($\mathcal{O}(mn\min(m,n))$), whereas a TAP iteration requires only a $2r\times 2r$ SVD ($\mathcal{O}(mnr)$). The scaling study confirms this empirically. The per-iteration cost ratio of DR to TAP grows from order one to about $18$ as the grid is refined from $32\times 64$ to $256\times 512$ (Table~\ref{tab:scaling} and Figure~\ref{fig:scaling}(a)). With its step size tuned at each grid, DR converges in fewer iterations than TAP, but its much higher per-iteration cost dominates the total wall-clock time: the two methods are comparable on the smaller grids and TAP is about $4\times$ faster on the finest grid (Table~\ref{tab:conv} and Figure~\ref{fig:scaling}(b)), an advantage that widens under refinement since the cost gap scales as $m/r$. TAP is therefore the most cost-efficient method, increasingly so at scale, while the convex formulation offers global optimality and lets the correction rank adapt freely.

\paragraph{Constraint preservation.} The formulations also differ in which constraint the iterate satisfies before convergence. TAP keeps the orthogonality $\bm X_k\bm B=\bm 0$ exactly at every iterate, and the iterate stays low rank, whereas the convex methods that recover the primal variable by clipping (DR, dual FISTA~I, dual AGD, PR+ CG, L-BFGS) keep the nonnegativity $\bm A+\bm X_k\geqslant\bm 0$ exactly at every iterate. The dual FISTA variant~II is the mirror case, keeping $\bm X_k\bm B=\bm 0$ exactly instead. Both constraints hold at convergence, but the per-iterate property matters when the iteration is stopped early or embedded in a larger computation.

\paragraph{Use as a positivity limiter.} The last experiment (Section~\ref{sec:limiter}) embeds the correction as a positivity limiter in the conservative LoMaC low-rank Vlasov solver. Both the convex (DR) and the rank-constrained (TAP) corrections remove the negativity introduced by the SVD truncation while leaving the conserved moments untouched, since $\bm X\bm B=\bm 0$. The per-iterate constraint preservation above is decisive here: TAP, which keeps $\bm X_k\bm B=\bm 0$ exactly, preserves mass, momentum and energy to machine precision, whereas DR preserves them only to its orthogonality residual. The correction itself enforces nonnegativity to machine precision, so the only residual negativity comes from re-expressing the corrected field in the rank-bounded low-rank format, and its level is set by the re-truncation tolerance, reaching machine zero as that tolerance is tightened at the cost of a higher but still moderate rank. For this to succeed, the correction rank must be allowed to track the solution rather than be fixed in advance.

\section{Conclusion}

We developed optimization-based post-processing algorithms to recover nonnegativity in low-rank numerical solutions of Vlasov dynamics while preserving the macroscopic conservation laws pointwise. The conservation requirement is written as an orthogonality constraint on the correction term. We studied two optimization formulations. The first is the convex problem \eqref{F2-1 (homo)} based on squared nuclear norm minimization. For this problem, Theorem~\ref{thm:proximal} shows that the proximal operator can be computed through an implicit singular value thresholding equation with a threshold determined by bisection. We developed five algorithms for this formulation: Douglas--Rachford splitting, restarted dual FISTA, restarted dual accelerated gradient descent, dual PR+ conjugate gradient, and dual L-BFGS. The second is the rank-constrained non-convex formulation \eqref{F2-2}. For this formulation, we developed a tangent-space accelerated alternating projection algorithm that only requires a \(2r \times 2r\) SVD per iteration. The numerical results for the Landau damping test case show that all algorithms produce comparable correction quality. The TAP algorithm is the most cost-efficient because of its low per-iteration cost, and this advantage becomes more pronounced as the problem size increases. Among the convex algorithms, Douglas--Rachford splitting converges the fastest. The gradient-restarted dual methods give the best performance among the dual approaches. The convex formulation has the advantage of global optimality, while the TAP algorithm naturally preserves low-rank storage throughout the iteration. Finally, we demonstrated that the correction can be used as a positivity limiter inside the time-dependent LoMaC low-rank scheme, which conserves mass, momentum and energy, using either the convex DR algorithm or the rank-constrained TAP algorithm with its rank tracking the solution ($r=N+3$). On the validated, genuinely low-rank bump-on-tail run up to $T=30$, both limiters remove the truncation negativity. Because the correction rank is not capped a priori, enforcing nonnegativity raises the kept rank, and the residual negativity is set by the re-truncation tolerance, driven toward machine precision as that tolerance is tightened. Since TAP maintains the orthogonality constraint exactly at every iterate, the TAP limiter preserves all three conserved quantities to machine precision, whereas DR preserves them to its orthogonality residual. Both methods incur a modest, comparable overhead.


\section*{Declaration of AI-assisted technologies}
All original ideas, as well as codes, are attributed to the authors. During the preparation of this work, the authors used Anthropic's Claude Code to assist with mathematical discussions, further numerical implementation, and drafting of the manuscript. After using this tool, the authors carefully reviewed and edited the content as needed and take full responsibility for the final publication.

\section*{Declarations}

\noindent\textbf{Funding.}
X.\ Zhang is partially supported by NSF DMS-2208518. J.-M.\ Qiu acknowledges support provided by Air Force Office of Scientific Research FA9550-24-1-0254 and Department of Energy DE-SC0023164. S.\ Becker acknowledges support from the Department of Energy DE-SC0023346.

\smallskip
\noindent\textbf{Competing Interests.}
The authors have no competing interests to declare that are relevant to the content of this article.

\smallskip
\noindent\textbf{Data and Code Availability.}
The test data and the code used to generate the numerical results in this paper are available from the authors upon request.

\smallskip
\noindent\textbf{Author Contributions.}
All authors contributed to the conception of the study, the methodology, the numerical implementation, and the writing of the manuscript. All authors read and approved the final manuscript.

\bibliographystyle{abbrv}
\bibliography{ref}

\end{document}